\documentclass[reqno]{article}

\usepackage{a4,amssymb, amsmath, amsthm, epsf, epsfig}

\renewcommand{\(}{\left(}
\renewcommand{\)}{\right)}

\newtheorem{theo}{Theorem}[section]
\newtheorem{prop}{Proposition}[section]
\newtheorem{lemma}{Lemma}[section]
\newtheorem{cor}{Corollary\!\!}[section]

\newtheorem{ncor}{Corollary}

\theoremstyle{definition}

\newtheorem{df}{Definition}[section]
\newtheorem{ex}{Example}

\theoremstyle{remark}

\newtheorem{rem}{Remark\!\!}

\newtheorem{nrem}{Remark}

\newcommand{\beq}{\begin{equation}}
\newcommand{\eeq}{\end{equation}}
\newcommand{\bal}{\begin{align}}
\newcommand{\eal}{\end{align}}
\newcommand{\bals}{\begin{align*}}
\newcommand{\eals}{\end{align*}}
\newcommand{\barr}[1]{\begin{array}{#1}}
\newcommand{\earr}{\end{array}}

\newcommand{\bth}{\begin{theo}}
\newcommand{\bl}{\begin{lemma}}
\newcommand{\el}{\end{lemma}}
\newcommand{\bp}{\begin{prop}}
\newcommand{\ep}{\end{prop}}
\newcommand{\bdf}{\begin{df}}
\newcommand{\edf}{\end{df}}
\newcommand{\brem}{\begin{rem}}

\newcommand{\bnrem}{\begin{nrem}}
\newcommand{\enrem}{\end{nrem}}
\newcommand{\bex}{\begin{ex}}
\newcommand{\eex}{\end{ex}}
\newcommand{\bcor}{\begin{cor}}
\newcommand{\ecor}{\end{cor}}
\newcommand{\bncor}{\begin{ncor}}
\newcommand{\encor}{\end{ncor}}
\newcommand{\bpf}{\begin{proof}}
\newcommand{\epf}{\end{proof}}

\def\({\left(}
\def\){\right)}

\numberwithin{equation}{section}

\title{The maximum degree of planar graphs I. Series-parallel graphs}

\author{Michael Drmota\footnote{Technische Univerisit\"{a}t Wien, Institute of
Discrete Mathematics and Geometry, Wiedner Hauptstrasse 8–10,
A–1040 Wien, Austria. {\tt michael.drmota@tuwien.ac.at}
\newline
Research supported by the Austrian Science Foundation FWF, Project S9604}
 \and Omer Gim\'{e}nez\footnote{Universitat Polit\`{e}cnica de Catalunya,
Departament de Llenguatges i Sistemes Inform\`{a}tics,
 Jordi Girona 1--3, 08034
Barcelona, Spain. {\tt omer.gimenez@upc.edu}}
  \and  Marc Noy\footnote{Universitat
Polit\`{e}cnica de Catalunya, Departament de Matem\`{a}tica Aplicada II,
Jordi Girona 1--3, 08034 Barcelona, Spain. {\tt  marc.noy@upc.edu}
 \newline
 Research supported in part by Ministerio de Ciencia e Innovaci\'{o}n
 MTM2008-03020. }}
\date{}

\begin{document}

\maketitle

\begin{abstract}
We prove that the maximum degree $\Delta_n$ of a random
series-parallel graph with $n$ vertices satisfies $\Delta_n/\log n
\to c$ in probability, and $\mathbb{E}\, \Delta_n \sim c \log n$ for
a computable constant $c>0$. The same result holds for outerplanar
graphs.
\end{abstract}

\section{Introduction}\label{sec1}

All graphs in this paper are simple and labelled.  We recall that a
graph is series-parallel if it does not contain the complete graph
$K_4$ as a minor; equivalently, if it does not contain a subdivision
of $K_4$. Since both $K_5$ and $K_{3,3}$ contain a subdivision of
$K_4$, by Kuratowski's theorem a series-parallel graph is planar. An
outerplanar graph is a planar graph that can be embedded in the plane
so that all vertices are incident to the outer face. They are
characterized as those graphs not containing a minor isomorphic to (or
a subdivision of) either $K_4$ or $K_{2,3}$; hence they form a
subclass of series-parallel graphs.

In a previous paper \cite{DGN2} we determined the degree
distribution in series-parallel graphs. More precisely, we showed
that that the probability that a given vertex has degree $k$ in a
random series-parallel graph with $n$ vertices tends to a computable
constant $\overline d_k>0$ for all $k\ge1$, as $n$ goes to infinity.
In the present paper we use the ideas introduced in \cite{DGN2} and
develop new techniques in order to study the maximum degree in
random series-parallel graphs.

Our main result  is the following. Let $\Delta_n$
denote the maximum degree of a random series-parallel graphs with
$n$ vertices. Then
\[
\frac{\Delta_n}{\log n} \to c \qquad\mbox{in probability}
\]
and
\[
\mathbb{E}\, \Delta_n \sim c \log n
\]
for a certain constant $c> 0$. The same result holds for 2-connected
series-parallel graphs, and for outerplanar and 2-connected
outerplanar graphs with suitable values of $c$.
 %
The constant $c$ is always well-defined analytically  and can be
computed as $c = 1/\log (1/q)$, where $0<q<1$ is the exponential
base in the asymptotic expansion of the corresponding degree
distribution. We have
\begin{align*}
c &\approx 3.482774  \quad  \mbox{for series-parallel
graphs}, \\
c &\approx 1.035792  \quad  \mbox{for outerplanar graphs},\\
c &\approx 3.679771  \quad  \mbox{for $2$-connected series-parallel graphs},\\
c &\approx 1.134592  \quad  \mbox{for $2$-connected outerplanar
graphs}.
\end{align*}

This result was conjectured in \cite{BPS} and, as we are going to
see, is the natural result to expect in this context. McDiarmid and
Reed \cite{MR} have recently proved that the maximum degree
$\Delta_n$ in random planar graphs is of order $\Theta(\log n)$ with
high probability; it is thus natural to expect that  with high
probability $\Delta_n \sim c\log n$ also in this case; this will be
treated in a companion paper \cite{DGN3}. We remark that an
analogous result holds for planar maps \cite{gao-wormald}, counted
according to the number of edges; in this case much more is know,
since the authors obtain the full limit distribution of $\Delta_n$.

Intuitively the reason for  the $\Delta_n \sim c\log n$ estimate is
the following. Let $d_{n,k}$ denote the probability that a random
vertex in a random planar graph of size $n$ has degree $k$. Then it
is known (see \cite{DGN2}) that $d_{n,k}\to \overline d_k$ as
$n\to\infty$, where $\overline d_k$ is a sequence of positive
numbers that satisfy $\overline d_{k}\sim c k^\alpha q^k$ as
$k\to\infty$ (for computable constants $c$ and $q$). Thus, we can
expect that $d_{n,k}\approx c k^\alpha q^k$ holds for $n,k\to\infty$
(in a properly chosen range) and also
\[
\sum_{\ell > k} d_{n,\ell}\approx \frac {cq}{1-q} k^\alpha q^k.
\]
Furthermore, let $Y_{n,k}$ denote the random variable that counts
the number of vertices of degree $> k$ in a random planar graph of size $n$. Then
\begin{equation}\label{eqYnkest}
\mathbb{E}\, Y_{n,k} = n\sum_{\ell > k}  d_{n,\ell} \approx n \frac
{cq}{1-q} k^\alpha q^k,
\end{equation}
and by definition
\[
Y_{n,k} > 0 \quad \Longleftrightarrow \quad \Delta_n> k.
\]
Hence the probability $\mathbb{P}\{ \Delta_n> k\} =
\mathbb{P}\{ Y_{n,k} > 0\} \le \mathbb{E}\, Y_{n,k}$
is negligible if  $n k^\alpha q^k \to 0$. Usually such a threshold
is tight so that one can expect that the converse statement is also
true, which implies 
$\Delta_n \sim c\log n$ for
$ c = 1/\log (1/q)$.

In this paper we make this heuristics rigorous by applying the first
and second moment method. The precise statement that we show is the
following, which can be considered as a kind of Master Theorem for
proving results on the maximum degree. The proof is based on
standard techniques, and we present it in Appendix~A.
\begin{theo}\label{Th1}
Let $d_{n,k}$ denote the probability that a randomly selected vertex
of a certain class of random graphs of size $n$ has degree $k$, and
let $d_{n,k,\ell}$ denote the probability that two different
randomly selected (ordered) vertices have degrees $k$ and $\ell$.
Suppose that we have the following properties:
\begin{enumerate}
\item There exists a limiting degree distribution $\overline d_k$ ($k\ge 1$)
with an asymptotic behaviour of the form
\begin{equation}\label{eqass1}
\log \overline d_k \sim k \log q \qquad (k \to\infty),
\end{equation}
where $q$ is a  real constant with $0<q<1$.
\item We have, as $n\to\infty$, $k\to\infty$, $\ell\to\infty$,
and uniformly for $k, \ell\le C \log n$
(for an arbitrary constant $C > 0$)
\begin{equation}\label{eqass2}
d_{n,k} \sim \overline d_k \quad \mbox{and}\quad d_{n,k,\ell} \sim \overline d_k \overline d_\ell.
\end{equation}
\item There exists $\overline q<1$ such that, uniformly for all $n,k,\ell\ge
1$,
\begin{equation}\label{eqass3}
d_{n,k} = O(\overline q^k)\quad \mbox{and}\quad
d_{n,k,\ell} = O(\overline q^{k+\ell}) .
\end{equation}
\end{enumerate}
Let $\Delta_n$ denote the maximum degree of a random graph of size
$n$ in this class. Then
\begin{equation}\label{eqDeltanresult0}
\frac{\Delta_n}{\log n} \to \frac 1{\log (1/q)}\qquad \mbox{in
probability}
\end{equation}
and
\begin{equation}\label{eqDeltanresult}
\mathbb{E}\, \Delta_n \sim \frac 1{\log (1/q)}\, \log n \qquad
(n\to\infty).
\end{equation}
\end{theo}
Condition 1 is fulfilled for planar, series-parallel and outerplanar
graphs, as shown in \cite{DGN2}. Condition 2 is the key to applying
the second moment method, as it gives access to the variance of the
number of vertices of degree $k$. It can be viewed as a kind of
asymptotic independence, in the sense that for random vertices $v_1$
and $v_2$, as $n\to\infty$ we have
$$
\mathbb{P}(\deg(v_1)=k,\deg(v_2)=l) \sim \mathbb{P}(\deg(v_1)=k)\,
\mathbb{P}(\deg(v_2)=l).
$$
The last condition is purely technical and is usually easy to
verify.

Most of the paper is devoted to showing that outerplanar and
series-parallel graphs satisfy the conditions imposed by
Theorem~\ref{Th1}, the bulk of the work being on verifying condition
2. For each of the two classes of graphs, we compute first the
associated counting generating functions from combinatorial
decompositions: this is done in Sections~\ref{secOutcomb}
(outerplanar) and~\ref{secSPcomb} (series-parallel). Then, we
analyze the generating functions as functions of complex variables
in order to obtain precise asymptotics for the probability that a
vertex or a pair
of vertices have given degrees: 
this is done in Sections~\ref{secOutAsymp} and~\ref{secSPAsymp}. These
sections make use of several technical lemmas whose proofs are based
on the Cauchy integration formula of analytic functions in several
variables. The proofs of these lemmas are given in Appendix~B.

Before concluding this introduction, we present some technical
preliminaries needed in the paper: first the combinatorics of
generating functions and secondly some analytic considerations.

\paragraph{Generating functions.}

We use the following notation. For a class of labelled graphs
$\mathcal{G}$ having $g_n$ graphs with $n$ vertices, we use
\[ G(x) = \sum_n g_n \frac{x^n}{n!} \]
to denote the \emph{exponential generating function} of
$\mathcal{G}$. A \emph{rooted graph} is a graph with a distinguished,
or marked, vertex. A \emph{double rooted graph} is a graph with two
marked vertices that are different (we cannot mark the same vertex
twice) and distinguishable (there is a first root vertex and a second
root vertex). For convenience, we assume that marked vertices are not
labelled, and they do not contribute towards the size of the
graph. Note that the derivatives of $G(x)$,
\[
  G'(x) = \sum_{n\geq 1} n g_n \frac{x^{n-1}}{n!} = \sum_{n} g_{n+1} \frac{x^{n}}{n!},
  \qquad G''(x) = \sum_{n} g_{n+2} \frac{x^{n}}{n!},
\]
can also be interpreted as the exponential generating functions of
rooted and double rooted graphs in $\mathcal{G}$.

When rooting a graph, we are often interested in the degree of
the marked vertices, so we introduce generating functions
\[ G^{\bullet}(x,w) = \sum_{n,k} g^\bullet_{n} d_{n+1,k} \frac{x^n}{n!} w^k,
\qquad G^{\bullet \bullet}(x,w,t) = \sum_{n,k,\ell} g^{\bullet
\bullet}_{n} d_{n+2,k,\ell} \frac{x^n}{n!} w^k t^\ell  \]
to enumerate rooted and double rooted graphs of $\mathcal{G}$, where
the exponent of variable $x$ counts the number of non-root vertices,
the exponents of variables $w$ and $t$ count the degrees of the
first and second root, and the numbers $d_{n,k}$ and
$d_{n,k,\ell}$ are the probabilities that a randomly selected vertex
(respectively, two randomly selected vertices) of a graph of
$\mathcal{G}$ of size $n$ has degree $k$ (respectively, have degrees
$k$ and $\ell$). Notice that with this terminology, the coefficient
$g^\bullet_{n} d_{n+1,k}$ is precisely the number of rooted graphs
with $n$ vertices and where the root has degree $k$, and similarly for
$ g^{\bullet \bullet}_{n} d_{n+2,k,\ell}$. Since it holds that
\[ G^{\bullet}(x,1) = G'(x), \qquad G^{\bullet \bullet}(x,1,1) = G''(x) \]
we see that $g^\bullet_n=g_{n+1}$ and $g^{\bullet
\bullet}_n=g_{n+2}$.

For the sake of readability, we always use $B(x)$ and $C(x)$ to
denote the generating functions of 2-connected and connected graphs
of the class of graphs we are working with, and $d_{n,k}$ and
$d_{n,k,\ell}$ to denote the associated probabilities; the context
will always make clear which class of graphs we refer to.

\paragraph{Singularity analysis.}

To obtain the asymptotic estimates for $d_{n,k}$ and $d_{n,k,l}$ as
required by Theorem~\ref{Th1}, we analyse the singularities of the
corresponding generating functions. It turns out that all these
generating functions share the same singularity structure, which we
proceed to describe.

A \emph{power series of the square-root type} is a power series
$y(x)$ with a square root singularity at
$x_0>0$, that is, $y(x)$ admits a local representation of the form
\begin{equation}\label{eqyx}
y(x) = g(x) - h(x)\sqrt{1-x/x_0}
\end{equation}
for $|x-x_0|<\varepsilon$ for some $\varepsilon>0$ and $|\arg(x-x_0)|>
0$, where $g(x)$ and $h(x)$ are analytic and non-zero at
$x_0$. Moreover, $y(x)$ can be analytically continued to the region
\[
D = D(x_0,\varepsilon) = \{ x\in \mathbb{C} : |x|< x_0+\varepsilon \}
\setminus [x_0,\infty).
\]
Note that $y(x)$ can be represented alternatively as a power series
in $X=\sqrt{1-x/x_0}$,
\[
y(x) = a_0 + a_1 X + a_2 X^2 + a_3 X^3 + \cdots
\]
for $|X| < \varepsilon^{1/2}$.

We illustrate the common singularity structure of our generating
functions by using the explicit expression of the generating function
of rooted 2-connected outerplanar graphs,
\begin{equation}\label{eq:B_bullet_xwy}
B^{\bullet}(x,w) = xw +      {xw^2 \over 2} {y(x) \over 1 -
    y(x)w},
\end{equation}
to be derived in Lemma~\ref{Leop21}, where $y(x)$ is an explicit
power series of the square-root type. Clearly, $B^\bullet(x,w)$ has
two possible sources of singularities: the square-root singularity
of $y(x)$ at $x=x_0$, and the vanishing of the denominator $1-y(x)w$
at $w=1/y(x)$. These two sources coalesce at the critical point
$x=x_0, w=1/y(x_0)$ (equivalently, $w=1/g(x_0)$ due to the local
representation $y(x)=g(x)-h(x)X$). We derive asymptotic estimates
for $d_{n,k}$ with $n,k\rightarrow \infty$, on the range $k \le
C\log n$ for any $C>0$, by using multivariate Cauchy coefficient
extraction on $B^\bullet(x,w)$ with an integration path close to the
critical point $(x,w)=(x_0, 1/g(x_0))$. More precisely, we integrate
along Hankel contours, following Flajolet and Odlyzko's transfer
theorems~\cite{FO}.

As for the remaining generating functions, they do not admit explicit
expressions as $B^\bullet(x,w)$, but they share the same singularity
structure: the square-root singularity of $y(x)$, and the vanishing of
a term $1-y(x)w$. Curiously enough, the nature of the singularity
induced by $1-y(x)w=0$ is different from case to case. This fact
justifies the need of distinct but closely related technical lemmas
tailored to the particular shapes of the generating functions. To wit,
these singularities are poles and double poles for 2-connected
outerplanar graphs (Lemmas~\ref{Le3} and \ref{Le4}), an essential
singularity $e^{\infty}$ for connected outerplanar graphs
(Lemmas~\ref{Le5} and \ref{Le6}, whose proofs make use of Hayman's
method~\cite{Hayman} to obtain the asymptotics), and powers of
square-roots for 2-connected and connected series-parallel graphs
(Lemmas~\ref{Le7} and \ref{Le8}).

\section{Outerplanar graphs: combinatorics}\label{secOutcomb}

In this section we study the combinatorics of rooted and double
rooted outerplanar graphs with respect to the degree of the roots.
We obtain explicit or nearly explicit expressions for the
corresponding generating functions, both for the case of 2-connected
and connected outerplanar graphs. The analysis of their
singularities is done later in Section~\ref{secOutAsymp}, where
we also state the main results on the maximum degree of outerplanar
graphs.

\subsection{$2$-Connected Outerplanar Graphs}

Recall that an outerplanar graph is a planar graph that can be
embedded in the plane so that all vertices are incident to the
external face. Furthermore, $2$-connected outerplanar graphs are
quite close to dissections. A dissection is a $(n+2)$-gon (with
$n\ge 1$) where one edge is rooted and the vertices are connected
inside it by means of diagonals that do not cross.

Let $A(x)$ denote the generating function of dissections
with $n+2$ vertices, that is, the two vertices of the
root edge are not counted. Then $A(x)$ is determined by the system
\begin{align*}
A(x) &= (1+A(x))\,x \, S(x),\\
S(x) &= (1+A(x))(1 + xS(x)),
\end{align*}
where the  generating function $S(x)$ enumerates non-empty {\it
chains} of dissections and single edges, the root edges
are lined up, and the two vertices of the chain (the two poles of
the network) are not counted (see Figure~\ref{f:dissections}).
Note that the term $1$ in $1+A(x)$ denotes the graph consisting of
a single edge (which is not a dissection) while the term $1$ in
$1+xS(x)$ denotes the empty chain.

\begin{figure}[htp] \centering
\includegraphics[width=0.45\hsize]{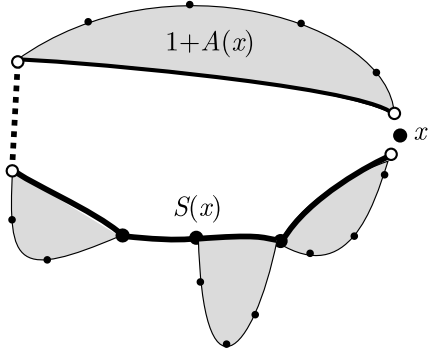}
\caption{Closing a chain of dissections and edges to obtain a new
  dissection; the dashed edge is the new root edge.}
          \label{f:dissections}
\end{figure}

This system can be solved explicitly  and we find that  $A(x)$
satisfies the quadratic equation
\[
2xA^2 + (3x-1)A + x = 0
\]
and is given by
\begin{equation}\label{eq:D}
    A(x) = {1-3x - \sqrt{1-6x+x^2} \over 4x}.
\end{equation}
Observe that
\[
S(x) = 2A(x) + 1 = {1-x - \sqrt{1-6x+x^2} \over 2x}.
\]

Next we consider 2-connected labelled outerplanar graphs. There is
exactly one 2-connected outerplanar graph with two vertices,
namely a single edge. However, when $n\ge 3$ the number of
2-connected outerplanar graphs is
\[
b_n =\frac {(n-1)!}2 \,a_{n-2}.
\]
To see this, note that $b_n$ coincides with the number
$b^\bullet_{n-1}$ of \emph{rooted} 2-connected outerplanar graphs on
$n$ vertices with labels $1,2,\ldots,n-1$ (choose the vertex with
label $n$ as root, and remove the label). Now consider a dissection
with $n$ vertices, whose number is $a_{n-2}$. Direct the root edge
of the dissection in counterclockwise order, and mark its first
vertex. Then, there are exactly $(n-1)!$ ways to label the remaining
$n-1$ vertices with $1,2,\ldots,n-1$. Finally, since the direction
of the outer cycle is irrelevant, divide the resulting number
$(n-1)! a_{n-2}$ by $2$ to get back $b_n$.
Now, it is just a matter of computation to obtain that
\[
B'(x) = \sum_{n\ge 1} b_{n+1}
\frac{x^n}{n!} = x + \frac 12 x A(x) = \frac {1+5x - \sqrt{1-6x+x^2}}{8}.
\]

Next we discuss the distribution of the degree of the first
root in dissections. It will be easy to translate this into a
corresponding result for $2$-connected outerplanar graphs. Let
$A(x,w)$ be the generating function of dissections of an
$(n+2)$-gon, where the exponent of $w$ counts the degree of the
first vertex of the root edge.  Similarly to the above we have
\[
A(x,w) = w(w+A(x,w))\, x\, S(x),
\]
and thus
\[
A(x,w) = \frac{xw^2 S(x)}{1-xS(x)} =  \frac{xw^2 (2A(x) + 1)}{1-xw(2A(x) + 1)}.
\]
In this context we introduce for later use a generating function
$S(x,w)$ that corresponds to series of dissections (compare with the
above description) where the exponent of $w$ counts the degree of the
first pole. Here we have
\[
S(x,w) = (w+A(x,w))(1+xS(x)).
\]

With the help of $A(x,w)$ we obtain an explicit representation for
the generating function for rooted $2$-connected outerplanar graphs.

\begin{lemma}\label{Leop21}
Let $B^\bullet(x,w)$ denote the bivariate generating function of
labelled rooted $2$-connected outerplanar graphs, where the exponent
of $x$ counts the number of non-root vertices and the exponent of
$w$ counts the degree of the root vertex. Then we have
\begin{align*}
B^\bullet(x,w) &= \sum_{n,k\ge 1} b_n^\bullet\, d_{n+1,k}\, \frac{x^n}{n!}\, w^ k \\
&= xw + \frac 12 x A(x,w) \\
&= xw +      {xw^2 \over 2} {x(2A(x)+1) \over 1 -
    x(2A(x)+1)w},
\end{align*}
where $b_n^\bullet=b_{n+1}$, and $d_{n+1,k}$ is the probability that
a randomly selected vertex of a labelled 2-connected outerplanar
graph of size $n+1$ has degree $k$.
\end{lemma}

Next we study double rooted 2-connected outerplanar graphs. As
before, we start with dissections. Let $A_1(x,w,t)$ denote the
generating function of dissections, where the exponents of $w$ and
$t$ count the degree of the first and the second vertex of the root
edge (in counterclockwise order). Similarly, we define $A_2(x,w,t)$
as the generating function of dissections $D$ with an additional
root vertex $v$ not in the root edge, where the exponent of $w$ and
$t$ count the degree of the first vertex of the root edge and the
degree of $v$, respectively. We also introduce the generating
function $S_2(x,w,t)$ for series of dissections with an additional
root $v$ not in the poles, where the exponents of $w$ and $t$ count
the degree of the first pole and the degree of $v$. Then we have the
following relations:
\begin{align*}
A_1(x,w,t) &= (w+A(x,w))wxt S(x,t)\\
A_2(x,w,t) &= (wt+A_1(x,w,t))wxS(x,t)\\
&+ A_2(x,w,t)xw S(x) \\
&+ (w+A(x,w))xw S_2(x,1,t),\\
S_2(x,w,t) &= A_2(x,w,t)(1+xS(x)) \\
&+ (wt+A_1(x,w,t))x S(x,t)\\
&+ (w+A(x,w))x S_2(x,1,t).
\end{align*}
The three summands for $A_2$ and $S_2$ correspond to the three
places where the additional root $v$ can be placed: inside the first
dissection, at the articulation vertex, or inside the series of
dissections. Summing up, this yields
\begin{align*}
A(x,w) &=  \frac{xw^2 (2A + 1)}{1-xw(2A + 1)},\\
S(x,w) &= \frac{w(1+x(2A+1))}{1-xw(2A + 1)},\\
A_1(x,w,t) &=  \frac{xw^2t^2(1+x(2A+1))}{(1-xw(2A + 1))(1-xt(2A + 1))} \\
A_2(x,1,t) &= \frac{xt^2(1+x(2A+1))}{(1-xt(2A+1))^2(1-x(4A+3))}\\
S_2(x,1,t) &= 2\frac{xt^2(1+x(2A+1))}{(1-xt(2A+1))^2(1-x(4A+3))}\\
A_2(x,w,t) &= \frac{xw^2t^2 (1+x(2A+1))( P_1   + x(wt-w-t)P_2)}{(1-xw(2A+1))^2(1-xt(2A+1))^2(1-x(4A+3))},
\end{align*}
where
\begin{align*}
P_1 &= 1-x(4A+1), \qquad  P_2 = 1-2A +x(2A+1).
\end{align*}

\begin{lemma}\label{Leop22}
Let $B^{\bullet\bullet}(x,w,t)$ denote the generating function of
double rooted labelled $2$-connected outerplanar graphs, where the
exponent of $x$ counts the number of non-root vertices, and the
exponents of $w$ and $t$ count, respectively, the degree of the
first and second root. Then we have
\begin{align*}
B^{\bullet\bullet}(x,w,t) &=
\sum_{n,k,\ell}
b_n^{\bullet\bullet}\, d_{n+2,k,\ell}\, \frac{x^n}{n!}\, w^k\, t^\ell\\
&= wt + \frac 12 A_1(x,w,t) + \frac 12 A_2(x,w,t)\\
&= wt + \frac{x(1+x(2A+1)t^2w^2}{2(1-xw(2A + 1))(1-xt(2A + 1))} \\
&+ \frac{xw^2t^2 (1+x(2A+1))( P_1   + x(wt-w-t)P_2)}{2(1-xw(2A+1))^2(1-xt(2A+1))^2(1-x(4A+3))},
\end{align*}
where $b_n^{\bullet\bullet} = b_{n+2}$, and $d_{n+2,k,\ell}$ denotes
the probability that two randomly selected vertices of a labelled
2-connected outerplanar graph of size $n+2$ have degrees $k$
and~$\ell$.
\end{lemma}

\subsection{Connected outerplanar graphs}

We recall that in many classes of graphs, including outerplanar,
series-parallel and planar graphs, a recursive relation holds
between the generating functions $B(x)$ and $C(x)$ of $2$-connected
and connected graphs, namely:
\begin{equation} \label{eqCpx}
C'(x) = e^{B'(xC'(x))}.
\end{equation}
This follows from the block decomposition of a connected graph; see,
for instance,~\cite{BGKN-series-parallel}. This relation can be
extended to the following one.

\begin{lemma}\label{lem:CpandCpp-outerplanar-combinatorics}
Let $C^\bullet(x,w)$ and $C^{\bullet\bullet}(x,w,t)$ be the generating
functions of rooted and double rooted connected outerplanar graphs,
where the exponents of $w$ and $t$ count the degree of the first and
second root. Then, we have
\[
C^\bullet(x,w) = e^{ B^\bullet(x C'(x),w)}
\]
and
\begin{align*}
C^{\bullet\bullet}(x,w,t) &=
 \frac{x}{(xC'(x))'} \frac{\partial}{\partial x}C^\bullet(x,w)
\frac{\partial}{\partial x} C^\bullet(x,t) \\
& +  B^{\bullet\bullet}(x C'(x),w,t)   C^\bullet(x,w)C^\bullet(x,t).
\end{align*}
\end{lemma}

\begin{proof}

The equation for $C^\bullet(x,w)$ is the natural extension of Equation
(\ref{eqCpx}), since the degree of the root vertex is the sum of the
degrees of the root vertices of the blocks incident to it.

\begin{figure}[htp] \centering
\includegraphics[width=0.8\hsize]{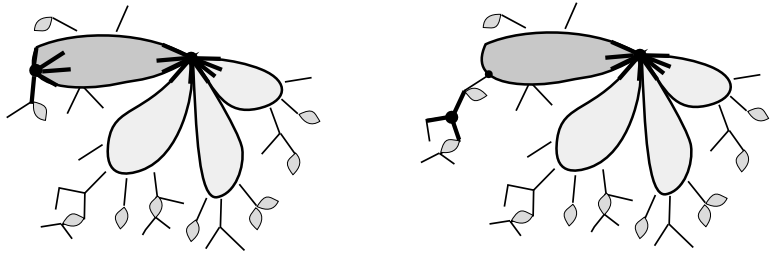}
\caption{Block decomposition of a connected graph with two roots, either
sharing a common block (left), or not (right).}
          \label{f:tworoots-connected}
\end{figure}

Next we consider double rooted graphs enumerated by $C^{\bullet
  \bullet}(x,w,t)$. Here we distinguish two situations, depending on
whether the two roots are in the same block or not (see
Figure~\ref{f:tworoots-connected}). In the former case, we have a
block with two roots, so a term $B^{\bullet \bullet}(x,w,t)$
appears. Each non-root vertex of this block is replaced by a rooted
connected graph in $xC'(x)$, while the second root is replaced by
$C^\bullet(x,t)$, since this replacement contributes to the degree of
the second root. Thus, the generating function of double rooted
connected graphs where the two roots are in the same block is
\[ e^{ B^\bullet(x C'(x),w)}B^{\bullet\bullet}(x C'(x),w,t) C^\bullet(x,t). \]

If the roots are in different blocks, it is still true that one of the
blocks incident to the first root is distinguished (the unique block
leading to the second root), and so is one of its vertices (the
articulation vertex leading to the second root). So this time the
decomposition has a rooted block with an additional distinguished
vertex, that is, a term $x\frac{\partial}{\partial x}
B^\bullet(x,w)$. Every vertex is replaced by $xC'(x)$, except for the
distinguished one, which is replaced by $x\frac{\partial}{\partial
  x}C^\bullet(x,t)$, that is, a rooted connected graph with an
additional root.

Hence, the generating function of $C^{\bullet \bullet}(x,w,t)$ is given by
\begin{align*}
C^{\bullet\bullet}(x,w,t) &=  e^{ B^\bullet(x C'(x),w)}
x \frac{\partial}{\partial x} B^\bullet(xC'(x),w)
\frac{\partial}{\partial x} C^\bullet(x,t) \\
& +  e^{ B^\bullet(x C'(x),w)}B^{\bullet\bullet}(x C'(x),w,t) C^\bullet(x,t)\\
 &=
 \frac{x}{(xC'(x))'} \frac{\partial}{\partial x}C^\bullet(x,w)
\frac{\partial}{\partial x} C^\bullet(x,t) \\
& +  B^{\bullet\bullet}(x C'(x),w,t)   C^\bullet(x,w)C^\bullet(x,t).
\end{align*}

\end{proof}


Note that in the outerplanar case the
functions $C^\bullet(x,w)$ and $C^{\bullet\bullet}(x,w,t)$ are
explicit in terms of $C'(x)$. For example, we have
\begin{equation}\label{eqCxw}
C^\bullet(x,w) = e^{xC'(x)w + a(x)w^2/(1-wb(x))},
\end{equation}
where
\begin{align*}
a(x) &= \frac 12 x^2 C'(x)^2(1 + 2A(xC'(x))),\\
b(x) &= xC'(x)(1 + 2A(xC'(x))).
\end{align*}
This is due to the fact that we have an explicit expression for
$B^\bullet(x,w)$, which is not true in general.

\section{Outerplanar graphs: asymptotics}\label{secOutAsymp}


In this section we analyze the singularities of the multivariate
generating functions derived in the previous section. By studying
the ``shape'' of these functions when $x,w$ and $t$ get close to the
relevant singularities we derive asymptotic estimates for the number
$d_{n,k}$ and $d_{n,k,\ell}$, as $n,k,\ell \to \infty$ in a suitable
range. By applying the Master Theorem discussed in the Introduction
we obtain a precise estimate for the maximum degree of outerplanar
graphs.

\subsection{$2$-Connected Outerplanar Graphs}

In Propositions \ref{Pro1} and \ref{Pro2} we obtain asymptotic
estimates for the numbers $d_{n,k}$ and $d_{n,k,\ell}$.

\begin{prop}\label{Pro1}
Let $d_{n,k}$ denote the probability that a randomly selected vertex
in a $2$-connected outerplanar graph with $n$ vertices has degree
$k$. Then we have uniformly for $k\le C \log n$
\begin{equation}\label{eqPro11}
d_{n,k} = 2 k(\sqrt 2 - 1)^k
\left(1 + O\left( \frac 1k \right) \right).
\end{equation}
Furthermore, we have uniformly for all $n,k\ge 1$
\begin{equation}\label{eqPro12}
d_{n,k}= O(\overline q^ k)
\end{equation}
for some real $\overline q$ with $0< \overline q < 1$.
\end{prop}

In order to prove the above result we need to perform singularity
analysis on the generating function $B^\bullet(x,w)$ given in
Lemma~\ref{Leop21}. The precise technical result we need is the
following.

\begin{lemma}\label{Le3}
Let $f(x,w) = \sum_{n,k} f_{n,k} x^n w^k$ be a bivariate generating
function of non-negative numbers $f_{n,k}$, and suppose that
$f(x,w)$ can be represented as
\begin{equation}\label{eqLe31}
f(x,w) = \frac{G(x,X,w)}{1-y(x)w},
\end{equation}
where $X=\sqrt{1-x/x_0}$, $y(x)$ is a power series with non-negative
coefficients of the square-root type as in (\ref{eqyx}),
\[ y(x) = g(x) - h(x)\sqrt{1-x/x_0}, \]
%
and the function $G(x,v,w)$ is analytic in the region
\[
D' = \{(x,v,w)\in \mathbb{C}^3 : |x|< x_0 + \eta, |v| < \eta, |w| < 1/g(x_0) + \eta\}
\]
for some $\eta > 0$, and satisfies $G(x_0,0,1/g(x_0))\ne 0$.

Then we have uniformly for $k \le C \log n$ (with an arbitrary constant $C > 0$)
\begin{equation}\label{eqLe31-asmp-exp}
f_{n,k} =  \frac{G(x_0,0,1/g(x_0)) h(x_0)}{2 \sqrt{\pi}} \,
g(x_0)^{k-1} x_0^{-n} k\, n^{-\frac 32} \left( 1 + O\left( \frac 1k \right) \right).
\end{equation}
Moreover, for every $\varepsilon > 0$ we have uniformly for all $n,k\ge 0$
\begin{equation}\label{eqLe32}
f_{n,k} = O\left( (g(x_0)+\varepsilon)^{k}   x_0^{-n} n^{-\frac 32} \right).
\end{equation}

If $g(x_0) < 1$, then $f_n = \sum_k f_{n,k}$ is given asymptotically
by \begin{equation}\label{eqLe33} f_n = \frac 1{2\sqrt\pi} \left(
\frac{h(x_0)G(x_0,0,1)}{(1-g(x_0))^2}  -
\frac{G_v(x_0,0,1)}{1-g(x_0)} \right) x_0^{-n} n^{-\frac 32} \left(
1 + O\left( \frac 1n \right) \right),
\end{equation}
and for every $k$ the limit
\[
\overline d_k = \lim_{n\to\infty} \frac{f_{n,k}}{f_n} \]
exists.
\end{lemma}
\noindent The proof of Lemma~\ref{Le3} is given in Appendix~B.

We proceed to prove Proposition~\ref{Pro1}. We remind the reader
that the limit $\overline d_k = \lim_{n\to\infty} d_{n,k}$ exists
for all \emph{fixed} $k$, and that (see \cite{DGN2})
\[
\overline d_k = 2 (k-1)(\sqrt 2 - 1)^k.
\]

\begin{proof}[Proof of Proposition \ref{Pro1}]

Recall that the $d_{n,k}$ are encoded in the function
$B^\bullet(x,w)$ (see Lemma~\ref{Leop21}). Next observe that
$B^\bullet(x,w)$ has precisely the form of $f(x,w)$ in
Lemma~\ref{Le3}. In particular, we have that
\[
y(x) = x(2A(x) + 1) =  \frac{1-x-\sqrt{1-6x+x^2}}{2}
\]
is a power series with non-negative coefficients that admits the local
expansion
\begin{align*}
y(x) &= g(x) - h(x) \sqrt{1-\frac x{x_0}} \\
g(x) &= \frac{1-x}{2}\\
h(x) &= \frac{1}{2} \sqrt{1-x x_0},\\
\end{align*}
with $x_0=3-2\sqrt{2}$, and that $G(x,X,w)$ is given by
\[
 G(x,X,w) = xw+\frac{xw}4\left(xw-w+x\sqrt{1-x x_0}X\right).
\]
Hence, we just need to compute the evaluations
\begin{align*}
g(x_0) &= \sqrt{2}-1\\
h(x_0) &= \sqrt{3\sqrt{2}-4} \\
G(x_0, 0, 1/g(x_0)) &= \frac{3-2\sqrt{2}}{2(\sqrt{2}-1)} \\
G(x_0, 0, 1) &= \frac12(13-9\sqrt{2}) \\
G_v(x_0, 0, 1) &= \frac12(3-2\sqrt{2})\sqrt{3\sqrt{2}-4}
\end{align*}
to obtain the asymptotics for $f_{n,k}$ and $f_n$, in accordance
with Equations~(\ref{eqLe31-asmp-exp}) and~(\ref{eqLe33}). Now it is
just a matter of computation to check that
\[
d_{n,k} = \frac{f_{n,k}}{f_n} =  2 k(\sqrt 2 - 1)^k
\left(1 + O\left( \frac 1k \right) \right),
\]
as required by (\ref{eqPro11}). Also, (\ref{eqPro12}) follows
immediately from (\ref{eqLe32}).
\end{proof}

Next we turn to the case of double rooted graphs.

\begin{prop}\label{Pro2}
Let $d_{n,k,\ell}$ denote the probability that two different (and
ordered) randomly selected vertices in a $2$-connected outerplanar
graph with $n$ vertices have degrees $k$ and $\ell$. Then we have
uniformly for $2\le k,\ell \le C \log n$
\begin{equation}\label{eqPro13}
d_{n,k,\ell} = 4k\ell\,(\sqrt 2 - 1)^{k+\ell}
\left( 1 + O\left( \frac 1k + \frac 1\ell \right) \right).
\end{equation}
Furthermore, we have uniformly for all $n,k\ge 1$
\begin{equation}\label{eqPro14}
d_{n,k,\ell}= O(\overline q^ {k+\ell}),
\end{equation}
for some real number $\overline q$ with $0< \overline q < 1$.
\end{prop}

Again we need a  precise technical result, to be proved in Appendix~B.

\begin{lemma}\label{Le4}
Let $f(x,w,t) = \sum_{n,k,\ell} f_{n,k,\ell} x^n w^k t^\ell$ be a
triple generating function of non-negative numbers $f_{n,k,\ell}$,
and assume that $f(x,w,t)$ can be represented as
\begin{equation}\label{eqLe41}
f(x,w,t) = \frac{G(x,X,w,t)}{X\,(1-y(x)w)^2(1-y(x)t)^2},
\end{equation}
where $X = \sqrt{1-x/x_0}$, $y(x)$ is a power series of the
square-root type as in (\ref{eqyx}), and $G(x,v,w,t)$ is
non-zero and analytic at $(x,0,w,t) = (x_0,0,1/g(x_0),1/g(x_0))$.

Then we have uniformly for $k,\ell \le C \log n$ (with an arbitrary constant $C > 0$)
\[
f_{n,k,\ell} =  \frac{G(x_0,0,1/g(x_0),1/g(x_0))}{\sqrt{\pi}} \,
g(x_0)^{k+\ell} x_0^{-n} k\,\ell\, n^{-\frac 12} \left( 1 + O\left( \frac 1k + \frac 1\ell \right) \right).
\]
Moreover, for every $\varepsilon > 0$ we have uniformly for all $n,k\ge 0$
\begin{equation}\label{eqLe34}
f_{n,k,\ell} = O\left( (g(x_0)+\varepsilon)^{k+\ell} x_0^{-n} n^{-\frac 12} \right).
\end{equation}

If $g(x_0)<1$, then $f_n = \sum_{k,\ell} f_{n,k,\ell}$ is given
asymptotically by
\[
f_n = \frac{G(x_0,0,1,1)}{\sqrt{\pi}(1-g(x_0))^4 } x_0^{-n}  n^{-\frac 12} \left( 1 + O\left( \frac 1n \right)\right)
\]
and for every pair $(k,\ell)$ the limit
\[
d_{k,\ell} = \lim_{n\to\infty} \frac{f_{n,k,\ell}}{f_n} \]
exists.
\end{lemma}

\begin{proof}[Proof of Proposition \ref{Pro2}]

Recall that $d_{n,k,\ell}$ are encoded in the function
$B^{\bullet\bullet}(x,w,t)$, which is given explicitly in
Lemma~\ref{Leop22}. The result follows from a direct application of
Lemma~\ref{Le4}, since $B^{\bullet\bullet}(x,w,t)$ is exactly of the
form $f(x,w,t)$, with the same $y(x)=x(2A(x)+1)$ and $x_0 =
3-2\sqrt{2}$, as in the proof of Lemma~\ref{Le3}. Indeed, the
factors $(1-xw(2A+1))^2$ and $(1-xt(2A+1))^2$ in the denominator of
$B^{\bullet\bullet}(x,w,t)$ become $(1-y(x)w)^2$ and $(1-y(x)t)^2$
in $f(x,w,t)$, and the factor $(1-x(4A+3))$ transforms into
$\sqrt{1-x x_0}X$ (the term $\sqrt{1-x x_0}$ is analytic and
contributes to $G(x,X,w,t)$).

We just need to compute the evaluations
\begin{align*}
G(x_0, 0, 1/g(x_0), 1/g(x_0)) &= \dfrac{\sqrt2}{2\sqrt{3\sqrt{2}-4}} \\
G(x_0, 0, 1, 1) &= \dfrac{\sqrt2 (\sqrt{2}-1)^4}{2\sqrt{3\sqrt{2}-4}} \\
(1-g(x_0))^4 &= (2-\sqrt{2})^4
\end{align*}
and check that
\[ \dfrac{G(x_0,0,1/g(x_0), 1/g(x_0))}{G(x_0,0,1,1)(1-g(x_0))^4} = 4. \]
Hence,
\begin{align*}
d_{n,k,\ell} &= \frac{f_{n,k,\ell}}{f_n} =  4 k\ell
(\sqrt{2}-1)^{k+\ell} \left(1+O\left( \frac 1k + \frac 1\ell \right)
\right),
\end{align*}
as required by (\ref{eqPro13}). Also, (\ref{eqPro14}) follows
immediately from (\ref{eqLe34}).


\end{proof}

\begin{theo}\label{Th2COP}
Let $\Delta_n$ denote the maximum degree of a random
labelled $2$-connected outerplanar graph with $n$ vertices.
Then
\[
\frac{\Delta_n}{\log n} \to \frac 1{\log (\sqrt 2 +1)} \qquad\mbox{in probability}
\]
and
\[
\mathbb{E}\, \Delta_n \sim \frac{\log n}{\log (\sqrt 2 +1)}
\qquad (n\to\infty).
\]
\end{theo}

\begin{proof}
The proof is an application of Theorem~\ref{Th1} to the class of
2-connected outerplanar graphs. Condition 1 of the theorem is a
direct consequence of either the asymptotics $\overline d_k \sim
ck^{\alpha}q^k $ derived in \cite{DGN2}, or the asymptotics from
$d_{n,k}$ of Proposition~\ref{Pro1}. Conditions 2 and 3 follow from
Proposition~\ref{Pro1} (for the $d_{n,k}$), and
Proposition~\ref{Pro2} (for the $d_{n,k,\ell}$). The condition
$d_{n, k, \ell} \sim \overline d_k \overline d_\ell$ is easily
verified from both asymptotic estimates.
\end{proof}

\begin{rem}
Prior to the proof of Proposition~\ref{Pro1}, we mentioned that
$\overline d_k = \lim_{n\to\infty} d_{n,k} = 2 (k-1)(\sqrt 2 -
1)^k$. This relation can be verified easily by considering the
generating function
\[
\overline p(w) = \sum_{k\ge 2} \overline d_k w^k = \lim_{n\to\infty} \frac{[x^n] B^\bullet(x,w)}{[x^n] B'(x)}.
\]
By setting \[
H(x,w,z) = xw +  {xw^2 \over 2} {4z-3x \over 1 -
  (4z-3x)w}
\]
we have $B^\bullet(x,w) = H(x,w,B'(x))$, and consequently, by
\cite[Lemma 3.1]{DGN2}, \[ \overline p(w) = H_z(3-2\sqrt
2,w,B'(3-2\sqrt 2)) = \frac{2 w^2} {(1+\sqrt{2}-w)^2},
\]
from where it follows the explicit expression for $\overline d_k$ for $k\ge
2$. Similarly we can analyze $B^{\bullet\bullet}(x,w,t)$.  Define
$\overline d_{k,\ell} = \lim_{n\to\infty} d_{n,k,\ell}$ and
\[
\overline p(w,t) = \sum_{k,\ell \ge 2} \overline d_{k,\ell}
w^k t^\ell = \lim_{n\to\infty} \frac{[x^n] B^{\bullet\bullet}(x,w,t)}{[x^n] B''    (x)}.
\]
The analytic situation is a little bit different from the univariate one.
The asymptotic {\it leading term} comes from the factor
$1/(1-x(4A+3))  = 1/\sqrt{1-6x+x^2}$. Hence it follows that
\begin{align*}
\overline p(w,t) & = \left[\dfrac{\dfrac{xw^2t^2(1+x(2A+1))(P_1 + x(wt-w-t)P_2)}
{2(1-xw(2A+1))^2(1-xt(2A+1))^2}}{\dfrac{x(1+x(2A+1))(P_1 -xP_2)}
{2(1-x(2A+1))^2(1-x(2A+1))^2}}\right]_{x = 3-2\sqrt 2, A = 1/\sqrt 2} \\
&= \frac{2 w^2} {(1+\sqrt{2}-w)^2} \frac{2 t^2} {(1+\sqrt{2}-t)^2} = \overline p(w) \overline p(t).
\end{align*}
It particular it follows that $\overline d_{k,\ell} = \overline d_k \overline d_\ell$ for
all $k,\ell \ge 2$. The interpretation of this relation is that
the degrees of the two randomly chosen vertices are independent
in the limit (which is not unexpected).
Furthermore, this relation provides an {\it automatic proof} that
$d_{n,k,\ell} \sim \overline d_k \overline d_\ell$.

\end{rem}

\subsection{Connected Outerplanar Graphs}


Again we need asymptotic expansions for $d_{n,k}$ and
$d_{n,k,\ell}$, but in this case it is a little more involved due
to the presence of essential singularities in the associated
generating functions. We recall that $C^\bullet(x,w)$ and
$C^{\bullet\bullet}(x,w,t)$ are explicit in terms of $C'(x)$.

The numbers $c_n$ of vertex labelled outerplanar graphs satisfy
\cite{BGKN-series-parallel}
\begin{align*}
c_n &=c\cdot n^{-\frac 52} \rho^{-n} n! \left( 1+ O\left( \frac
1n\right)\right),
\end{align*}
where $\rho = v_0 e^{-B'(v_0)} \approx 0.136594$ is the radius of
convergence of $C(x)$ and  $v_0  \approx 0.170765$ satisfies the
equation $1 = v_0 B''(v_0)$. Furthermore, $C'(x)$ has a local
representation of the form (\ref{eqyx}).

\begin{prop}\label{Pro3}
Let $d_{n,k}$ denote the probability that a randomly selected
vertex in a connected outerplanar graph with $n$ vertices
has degree $k$. Then we have as $n,k\to\infty$ and
uniformly for $k\le C \log n$
\begin{equation}\label{eqPro31}
d_{n,k} \sim \overline d_k,
\end{equation}
where $\overline d_k$ denotes the asymptotic degree distribution of
connected outerplanar graphs encoded by the generating function
\begin{align*}
\overline p(w) &= \sum_{k\ge 1} \overline d_k w^k \\
&=\rho\frac{v_0^2(2A(v_0) +1)(2A(v_0) +1 +2v_0A'(v_0))w^2}
{2(1- v_0(2A(v_0)+1)w)^2} \\
&\qquad\times \exp\left( v_0 w + \frac{v_0^2(2A(v_0)+1)w^2 }
{2(1-v_0(2A(v_0)+1)w)}  \right)
\end{align*}
and is given asymptotically by
\[
\overline d_k \sim c_1 k^{ 1/4} e^{c_2 \sqrt k} q^k,
\]
where $c_1 \approx 0.667187$, $c_2 \approx 0.947130$.

Furthermore, we have uniformly for all $n,k\ge 1$
\begin{equation}\label{eqPro32}
d_{n,k}= O(\overline q^ k),
\end{equation}
for some real $\overline q$ with $0< \overline q < 1$.
\end{prop}

The corresponding technical result needed to prove the previous
proposition is the following.

\begin{lemma}\label{Le5}
Let $f(x,w) = \sum_{n,k} f_{n,k} x^n w^k$ be a bivariate generating
function of non-negative numbers $f_{n,k}$, and assume that $f(x,w)$
can be represented as
\begin{equation}\label{eqLe51}
f(x,w) = G(x,X,w) \exp\left(\frac{H(x,X,w)}{1-y(x)w}\right),
\end{equation}
where $X = \sqrt{1-x/x_0}$, $y(x)$ is a power series of the
square-root type as in (\ref{eqyx}), and the functions $G(x,v,w)$ and
$H(x,v,w)$ are non-zero and analytic at $(x,v,w) = (x_0,0,1/g(x_0))$.

Then we have uniformly for $k \le C \log n$ (with an arbitrary constant $C > 0$)
\begin{align*}
f_{n,k} &=  \frac{G\left(x_0,0,\frac 1{g(x_0)}\right) h(x_0) \,
H\left (x_0,0,\frac 1{g(x_0)}\right)^{\frac 14}}{4\pi} \,e^{\frac 12 H(x_0,0,1/g(x_0)) -
\frac 1{g(x_0)} H_w(x_0,0,1/g(x_0))}  \\
&\qquad \times g(x_0)^{k-1} x_0^{-n} k^{\frac 14} e^{2 \sqrt{Hx_0,0,1/g(x_0))\, k}} n^{-\frac 32} \left( 1 + O\left( \frac 1{\sqrt k} \right) \right).
\end{align*}
Moreover, for every $\varepsilon > 0$ we have uniformly for all $n,k\ge 0$
\[
f_{n,k} = O\left( (g(x_0)+\varepsilon)^{k}   x_0^{-n} n^{-\frac 32} \right).
\]

If $g(x_0) < 1$ then, $f_n = \sum_k f_{n,k}$ is given asymptotically
by \begin{align*} f_n &= \exp\left( \frac{H(x_0,0,1)}{1-g(x_0)}
\right) \left(  G(x_0,0,1) \left(
\frac{h(x_0)H(x_0,0,1)}{(1-g(x_0))^2}  -
\frac{H_v(x_0,0,1)}{1-g(x_0)}\right) - G_v(x_0,0,1) \right)\\
&\qquad \times x_0^{-n} n^{-\frac 32} \left( 1 + O\left( \frac 1n
\right) \right),
\end{align*}
and for every $k$ the limit
\[
\overline d_k = \lim_{n\to\infty} \frac{f_{n,k}}{f_n} \]
exists.
\end{lemma}

\begin{proof}[Proof of Proposition~\ref{Pro3}]

The existence of $\overline d_k$, the probability generating
function $\overline p(w)$ encoding them, and the asymptotic
expression $\overline d_k \sim c_1 k^{1/4} e^{c2 \sqrt{k}} q^k$,
have been derived in \cite{DGN2}. It only remains to show the
asymptotic relations (\ref{eqPro31}) and (\ref{eqPro32}); these will
follow from an application of Lemma~\ref{Le5} to
$f(x,w)=C^\bullet(x,w)$ with $x_0 = \rho$, the radius of convergence
of $C'(x)$.

Indeed, recall from Equation~(\ref{eqCxw}) that $C^\bullet(x,w)$ is of
the form
\[ C^\bullet(x,w) = \exp\(xC'(x)w\) \exp\(\frac{a(x)w^2}{1-wb(x)}\),
\]
with
\begin{align*}
 a(x) &= \frac12 x^2C'(x)^2(1+2A(xC'(x))), \\
 b(x) &= xC'(x)(1+2A(xC'(x))).
\end{align*}
Note that $C'(x)$ is not analytic at $x=\rho$. Thus, we must use a
local representation $C'(x)=\overline g(x)-\overline
h(x)\sqrt{1-x/\rho}$ of the form~(\ref{eqyx}) in order to obtain the
analytic expressions $A(x,v,w)$ and $G(x,v,w)$. In contrast with the
2-connected case, we are evaluating the function $A(x)$ at a point
$\rho C'(\rho)$ smaller that its radius of convergence
$3-2\sqrt{2}$, so that both $A(xC'(x))$ and
%
\[
y(x) = b(x) = xC'(x)\left( 1 + 2 A(xC'(x)) \right),
\]
admit a local representation of the form (\ref{eqyx}).

Hence, we can apply Lemma~\ref{Le5} and deduce (\ref{eqPro31}) and
(\ref{eqPro32}). Note that the asymptotic expansion for the
$\overline d_k$ is derived from two sources, on the one side from
asymptotic estimates on the coefficients of the PGF $\overline
p(w)$, as in \cite{DGN2}, and on the other side from the limit
$f_{n,m}/f_n$ from Lemma~\ref{Le5}. As expected, both asymptotic
expansions coincide.
\end{proof}

The estimates for double rooted graphs come next, together with
the associated technical result.

\begin{prop}\label{Pro4}
Let $d_{n,k,\ell}$ denote the probability that two different
(and ordered) randomly selected
vertices in a connected outerplanar graph with $n$ vertices
have degrees $k$ and $\ell$.
Then we have for $n,k,\ell\to\infty$ and uniformly for
$2\le k,\ell \le C \log n$
\begin{equation}\label{eqPro41}
d_{n,k,\ell} \sim \overline d_k \overline d_\ell,
\end{equation}
where $\overline d_k$ denotes the asymptotic degree distribution of
connected outerplanar graphs.
Furthermore, we have uniformly for all $n,k\ge 1$
\begin{equation}\label{eqPro42}
d_{n,k,\ell}= O(\overline q^ {k+\ell})
\end{equation}
for some real $\overline q$ with $0< \overline q < 1$.
\end{prop}

\begin{lemma}\label{Le6}
Let $f(x,w,t) = \sum_{n,k,\ell} f_{n,k,\ell} x^n w^k t^\ell$ be a
triple generating function of non-negative numbers $f_{n,k,\ell}$,
and suppose that $f(x,w,t)$ can be represented as
\begin{equation}\label{eqLe61}
f(x,w,t) = \frac{G(x,X,w,t)}{X }
\frac{\exp\left(\frac{H(x,X,w)}{1-y(x)w}+ \frac{H(x,X,t)}{1-y(x)t}\right)}
{(1-y(x)w)^2(1-y(x)t)^2},
\end{equation}
where $X = \sqrt{1-x/x_0}$, the functions $G(x,v,w,t)$ and
$H(x,v,w)$ are non-zero and analytic at $(x,v,w,t) =
(x_0,0,1/g(x_0),1/g(x_0))$, and $y(x)$ is a power series of the
square-root type as in (\ref{eqyx}).

Then we have uniformly for $k,\ell \le C \log n$ (with an arbitrary constant $C > 0$)
\begin{align*}
f_{n,k,\ell} &=  \frac{G(x_0,0,1/g(x_0),1/g(x_0)}
{4 \pi^{3/2} H(x_0,0,1/g(x_0))^{3/2}} e^{H(x_0,0,1/g(x_0)) - \frac 2{g(x_0)}H_w(x_0,0,1/g(x_0))}\\
&\times g(x_0)^{k+\ell} x_0^{-n} (k\ell)^{\frac 14} e^{2\sqrt{H(x_0,0,1/g(x_0)} (\sqrt k + \sqrt \ell)} n^{-\frac 12} \left( 1 + O\left( \frac 1{\sqrt k} + \frac 1{\sqrt \ell} \right) \right).
\end{align*}
Moreover, for every $\varepsilon > 0$ we have uniformly for all $n,k\ge 0$
\[
f_{n,k,\ell} = O\left( (g(x_0)+\varepsilon)^{k+\ell}  x_0^{-n} n^{-\frac 12} \right).
\]

If $g(x_0)<1$, then $f_n = \sum_{k,\ell} f_{n,k,\ell}$ is given
asymptotically  by
\[
f_n =  \frac{G(x_0,0,1,1)\exp\left( \frac{2 H(x_0,0,1)}{1-g(x_0)} \right)}
{\sqrt{\pi} (1-g(x_0))^4}\,
 x_0^{-n}  n^{-\frac 12} \left( 1 + O\left( \frac 1n \right)\right),
\]
and for every pair $(k,\ell)$ the limit
\[
d_{k,\ell} = \lim_{n\to\infty} \frac{f_{n,k,\ell}}{f_n} \]
exists.
\end{lemma}

\begin{proof}[Proof of Proposition~\ref{Pro4}]
As usual, we check that $C^{\bullet\bullet}(x,w,t)$ has the form of
the generating function $f(x,w,t)$ in Lemma~\ref{Le6}. Then
(\ref{eqPro41}) and (\ref{eqPro42}) will follow automatically.

From (\ref{eqCxw}) it follows that
\[
\frac{\partial}{\partial x}C^\bullet(x,w) = e^{xC'(x)w +
a(x)w^2/(1-wb(x))} \left(  (xC'(x))' w + \frac{a'(x)w^2}{1-wb(x)} +
\frac{a(x)b'(x)w^3}{(1-wb(x))^2} \right).
\]
By using the local expansion of $C'(x)$ it follows that
$(xC'(x))'$, $a'(x)$, and $b'(x)$ can be represented as
\[
\frac{\overline g(x) - \overline h(x) \sqrt{1 -x/\rho} } {\sqrt{1
-x/\rho}},
\]
with functions $\overline g(x), \overline h(x)$ that are analytic
and non-zero for $x = \rho$. Furthermore, observe that
$B^{\bullet\bullet}(x,w,t)$ is analytic for $x = v_0 = \rho
C'(\rho)$. Hence it  follows easily that $C^{\bullet\bullet}(x,w,t)$
satisfies the assumptions of Lemma~\ref{Le6}, as claimed.
\end{proof}

\begin{theo}\label{ThCOP}
Let $\Delta_n$ denote the maximum degree of a random
connected outerplanar vertex labelled graph with $n$ vertices.
Then
\[
\frac{\Delta_n}{\log n} \to  c \qquad\mbox{in probability}
\]
and
\[
\mathbb{E}\, \Delta_n \sim c\,\log n
\qquad (n\to\infty),
\]
where $c = \displaystyle\frac1{\log(1/q)}$, and $q$ is given by
\[
q = v_0\left( 1 + 2 A(v_0)\right) \approx 0.380813,
\]
and  $v_0 \approx 0.1707649$ satisfies the equation $1 = v_0
B''(v_0)$. 
\end{theo}

\begin{rem}
Similarly as in the $2$-connected case, it is possible to check the
relation $\overline d_{k,\ell} = \overline d_k\overline d_\ell$, or
equivalently $\overline p(w,t) = \overline p(w)\overline p(t)$.
However, in the connected case we can prove a more universal
property.  Suppose that we have a class of vertex labelled graphs
whose block decomposition translates into the equation $C'(x) =
e^{B'(xC'(x))}$, and consequently into $C^\bullet(x,w) = e^{
  B^\bullet(x C'(x),w)}$. Furthermore, assume that the radius of
convergence of $B'(x)$ is strictly larger than the evaluation of
$xC'(x)$ at its radius of convergence. Then we automatically have
the property that $C'(x)$ has a square-root singularity at its
dominant singularty $\rho$ (which is also the radius of
convergence).  Setting $H(z,w) = e^{B^\bullet(z,w)}$, we have that
(again by \cite[Lemma
  3.1]{DGN2}) \[ \overline p(w) = \sum_{k\ge 1} \overline d_k w^k = H_z(\rho C'(\rho),w) =
e^{ B^\bullet(\rho C'(\rho),w)} \frac{\partial B^\bullet}{\partial
  x}(\rho C'(\rho),w).
\]
Next observe that the asymptotic leading part of $C^{\bullet\bullet}(x,w,t)$ comes from the term
\[
T:= \frac{x}{(xC'(x))'} \frac{\partial}{\partial x}C^\bullet(x,w) \frac{\partial}{\partial x} C^\bullet(x,t).
\]
Since
\[
\frac{\partial}{\partial x}C^\bullet(x,w) = H_z (xC'(x),w)
(xC'(x))',
\]
we also have
\[
T = H_z(xC'(x),w)H_z(xC'(x),t) x(xC'(x))',
\]
and consequently
\[
\overline p(w,t) = \lim_{n\to\infty} \frac{[x^n]\, T}{[x^n] x(xC'(x))' } = H_z(\rho C'(\rho),w)H_z(\rho C'(\rho),t) = \overline p(w) \overline p(t).
\]
In particular we obtain the relation $\overline d_{k,l}=\overline
d_k \overline d_\ell$ for connected outerplanar graphs, as well as
for connected series-parallel graphs.
\end{rem}

\section{Series-parallel graphs: combinatorics}\label{secSPcomb}

We now turn our attention to the combinatorics of rooted and double
rooted 2-connected and connected series-parallel graphs with respect
to the degree of the roots. In this section we derive equations for
the generating functions of these families of graphs, while their
singularity analysis is postponed to Section~\ref{secSPAsymp}.

\subsection{SP networks}

Recall that a connected series-parallel graph can be seen as the
result of repeated  series-parallel edge extensions applied to a
tree. Thus, the basic element of a series-parallel graph is the
result of series-parallel edge extensions of a single edge. Such
graphs are also called series-parallel
networks\index{series-parallel network}. They have two distinguished
vertices (or roots) that are called {\it poles}\index{poles of
series-parallel networks}. Series-parallel extensions induces a
recursive description of SP networks: they are either a parallel
composition of SP networks, a series composition of SP networks, or
just the smallest network consisting of the two poles and an edge
joining them.

Let $E(x)$ and $S(x)$ be the generating functions for labelled SP
networks and series SP networks, where the exponent of $x$ counts the
number of vertices other than the two poles. They satisfy the
relations
\begin{align*}
E(x) &= 2e^{S(x)} - 1,\\
S(x) &= x(E(x)-S(x))E(x).
\end{align*}
The first equation follows from the fact that a series SP network is a
non-empty set of series SP networks (it is a parallel SP network if
the set contains more than one network, and a series SP network
otherwise), where the factor $2$ stands for choosing whether we have
an edge joining both poles or not (see
Figure~\ref{fig:SP-network}). The second equation stablishes that a
series SP network is always the series composition of two SP networks,
where the first one is taken to be non-series to avoid multiple
counting.

\begin{figure}[htp] \centering
\includegraphics[width=0.8\hsize]{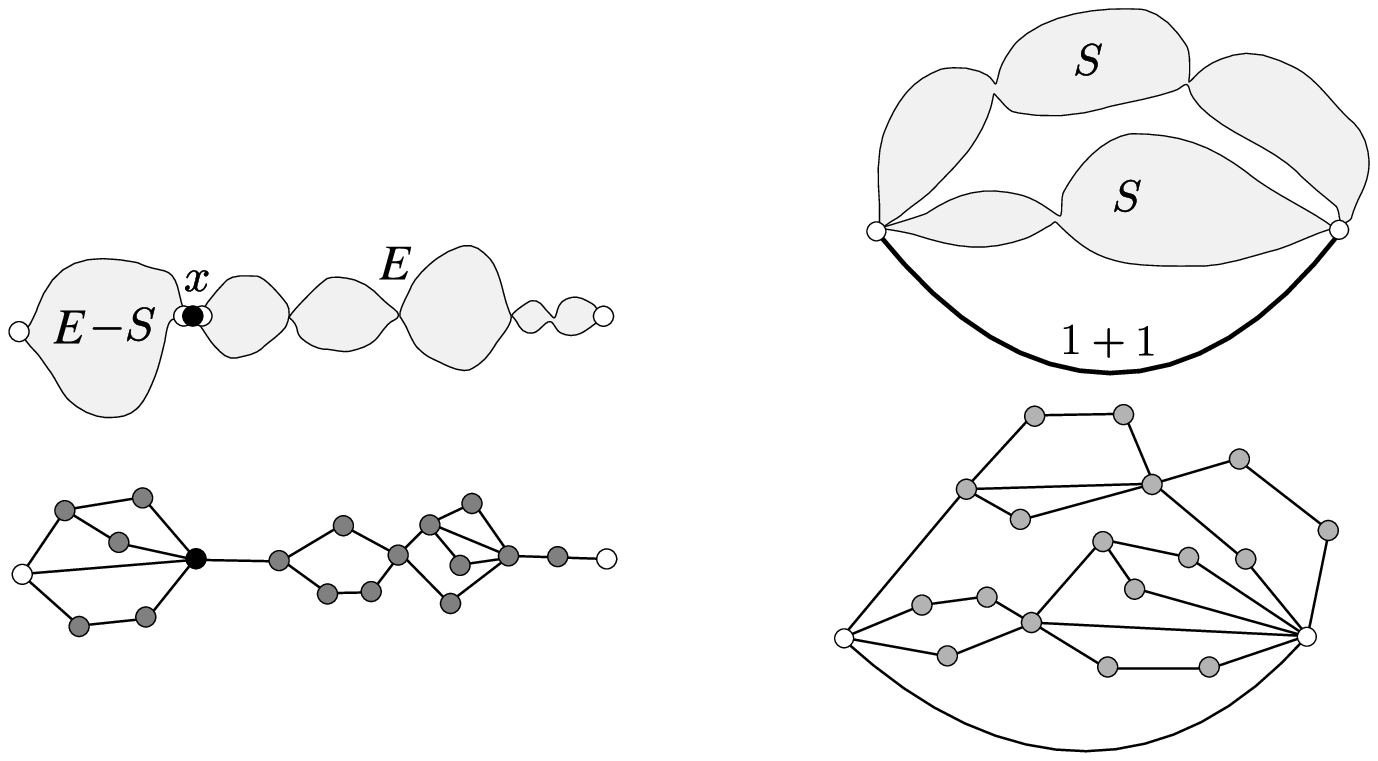}
\caption{A parallel composition (left) and a series composition (right) of SP networks.} \label{fig:SP-network}
\end{figure}

Next let $D(x,w)$ and $S(x,w)$ be the generating functions for SP
and series SP networks, where the exponent of $w$ counts the degree
of the first pole. Note that $E(x)=D(x,1)$ and $S(x)=S(x,1)$. (To be
consistent, we should have chosen $D(x)$ instead of $E(x)$, but we
have opted for $E(x)$ to avoid confusion in future formulas). Here
we have
\begin{align*}
D(x,w) &= (1+w)e^{S(x,w)} - 1,\\
S(x,w) &= x(D(x,w)-S(x,w))E(x).
\end{align*}

We now consider double rooted SP networks: the first root is taken
always as the first pole, while the second root is a vertex other
than the first pole. Let $D_1(x,w,t)$ and $S_1(x,w,t)$ denote the
generating functions for SP and series SP networks where the second
root is the second pole, and $D_2(x,w,t)$ and $S_2(x,w,t)$ denote
the generating functions for SP and series SP networks where the
second root is not the second pole. Here, the exponents of $w$ and
$t$ count the degree of the first and second root, respectively.
The corresponding relations are
\begin{align*}
D_1(x,w,t) &= (1+wt)e^{S_1(x,w,t)} -1,\\
S_1(x,w,t) &= x(D(x,w)-S(x,w))D(x,t),\\
D_2(x,w,t) &= (1+w)e^{S(x,w)}S_2(x,w,t), \\
S_2(x,w,t) &= x(D_2(x,w,t)-S_2(x,w,t))E(x)\\
&+ x(D_1(x,w,t)-S_1(x,w,t))D(x,t)\\
&+ x(D(x,w)-S(x,w))D_2(x,1,t).
\end{align*}

Note that $D(x,1) = D_1(x,1,1) = E(x)$, $S(x,1) = S_1(x,1,1) = S(x)$ and that
$D_2(x,1,1) = xE'(x)$ and $S_2(x,1,1) = xS'(x)$. Also, observe that, by symmetry, $D_1(x,1,t)=D(x,t)$.

These equations can be easily solved. First one uses the implicit equation
\begin{equation}\label{eqEx}
E(x) = 2 \exp\left(\frac{xE(x)^2}{1+xE(x)} \right) - 1
\end{equation}
for $E(x)$ to express
\[
S(x) = \frac {xE(x)^2}{1+xE(x)}.
\]
Secondly, the implicit equation
\begin{equation}\label{eqDxw-imp}
D(x,w) = (1+w) \exp\left(\frac{xD(x,w)E(x)}{1+xE(x)} \right) - 1
\end{equation}
determines $D(x,w)$, and it can be used to obtain
\[
S(x,w) = \frac{xE(x)}{1+xE(x)} D(x,w).
\]
With the help of these representations we get directly
\begin{align*}
D_1(x,w,t) &= (1+wt)\exp\left( \frac{x}{1+xE(x)} D(x,w)D(x,t) \right) -1,\\
S_1(x,w,t) &=  \frac{x}{1+xE(x)} D(x,w)D(x,t).
\end{align*}
Next we set $w=1$ and obtain from the two equations for
$D_2$ and $S_2$ the representations
\begin{align*}
D_2(x,1,t) &= \frac{x(1+E(x))}{1-2xE(x)^2-x^2E(x)^3} D(x,t)^2,\\
S_2(x,1,t) &= \frac{x}{1-2xE(x)^2-x^2E(x)^3} D(x,t)^2.
\end{align*}
Finally this gives
\begin{align*}
D_2(x,w,t) &= \frac{x(1+D(x,w))D(x,t)}{1-xE(x)D(x,w)}
\left( (1+wt)\exp\left( \frac{x}{1+xE(x)} D(x,w)D(x,t) \right) -1 \right)\\
&+ \frac{x^2E(x)(1+xE(x))}{1-2xE(x)^2-x^2E(x)^3}
\frac{(1+D(x,w))D(x,w)D(x,t)^2}{1-xE(x)D(x,w)},\\
S_2(x,w,t) &= \frac{xD(x,t)}{1-xE(x)D(x,w)}
\left( (1+wt)\exp\left( \frac{x}{1+xE(x)} D(x,w)D(x,t) \right) -1 \right)\\
&+ \frac{x^2E(x)(1+xE(x))}{1-2xE(x)^2-x^2E(x)^3}
\frac{D(x,w)D(x,t)^2}{1-xE(x)D(x,w)}.
\end{align*}

\subsection{$2$-connected SP graphs}

It has been shown (see \cite{DGN2}) that the generating function
$B(x)$ of $2$-connected SP graphs can be expressed in terms of
$E(x)$ as
\begin{equation}\label{eqBxEx}
B(x) = \frac12 \log(1 + xE(x)) - \frac{xE(x)(x^2E(x)^2 + xE(x) + 2 -
2x)}{4(1+xE(x))}.
\end{equation}
Next we recall that the generating function $B^\bullet(x,w)$ of
rooted $2$-connected SP graphs, where the exponent of $x$ counts the
number of non-root vertices and the exponent of $w$ the degree of
the root, satisfies
\begin{equation}\label{eqBxw0}
w\frac{\partial}{\partial w} B^\bullet (x,w) =
\sum_{k\ge 1} k B_k(x) w^k = xwe^{S(x,w)}.
\end{equation}
From this relation it is possible to obtain
an explicit representation for $B^\bullet(x,w)$.

\begin{lemma}[From {\cite[Lemma 4.2]{DGN2}}]\label{Lesp21}
Let $B^\bullet(x,w)$ be the generating function of
vertex rooted $2$-connected SP graphs, where the exponent of
$x$ counts the number of vertices and the exponent of $w$ the
degree of the root vertex. Then we have
\[
B^\bullet(x,w) = x\left(D(x,w) - {xE(x) \over 1+xE(x)} D(x,w)
\left(1 + {D(x,w) \over 2} \right)     \right).
\]
\end{lemma}

We also obtain an explicit expression for the generating function
of double rooted 2-connected graphs.

\begin{lemma}\label{Lesp22}
Let $B^{\bullet\bullet}(x,w,t)$ denote the generating function of
labelled double rooted $2$-connected SP graphs, where the exponent of
$x$ counts the number of non-root vertices, and the exponents of $w$
and $t$ count the degree of the two roots. Then we have
\begin{align}
& w\frac{\partial}{\partial w} B^{\bullet\bullet}(x,w,t)
= wt e^{S_1(x,w,t)} + we^{S(x,w)} S_2(x,w,t) \label{eqLesp22}\\
&= wt \exp\left(\frac{x}{1+xE(x)} D(x,w)D(x,t)\right) \nonumber \\
& +\frac{xwD(x,t)(D(x,w)+1)}{(1+w)(1-xE(x)D(x,w))} \left(
(1+wt)\exp\left(\frac{xD(x,w)D(x,t)}{1+xE(x)}\right)-1
\right. \nonumber \\
& \left. +\frac{xE(x)(1+xE(x))D(x,w)D(x,t)}{1-2xE(x)^2+x^2E(x)^3}\right) \nonumber
\end{align}
\end{lemma}

\begin{proof}
Equation~(\ref{eqLesp22}) is the natural extension of
Equation~(\ref{eqBxw0}) to double rooted graphs. Both follow from
the fact that we can obtain a SP network with non-adjacent poles
(that is, an object enumerated by $e^{S(x)}$) by distinguishing,
orienting and then deleting any edge of an arbitrary 2-connected SP
graph. Here, we use the partial derivative $\partial/\partial w$ to
distinguish an edge incident to the first root. Finally, we note
that the two summands of Equation~\ref{eqLesp22} correspond,
respectively,  to the case where the second root is the other vertex
of the distinguished edge,  and to the case where it is not.
\end{proof}

\subsection{Connected SP graphs}

As before we denote the corresponding generating functions
for connected SP graphs by
$C'(x)$, $C^\bullet(x,w)$, and $C^{\bullet\bullet}(x,w,t)$.
The function $C'(x)$ satisfies the same equation as
in the outerplanar case,
\[
C'(x) = e^{B'(xC'(x))}.
\]
Indeed, the situation is completely analogous to that of rooted and
double rooted connected outerplanar graphs (compare with
Lemma~\ref{lem:CpandCpp-outerplanar-combinatorics}).
\begin{lemma}\label{lem:CpandCpp-SP-combinatorics}
Let $C^\bullet(x,w)$ and $C^{\bullet\bullet}(x,w,t)$ be the generating
functions of rooted and double rooted connected SP graphs, where the
exponents of $w$ and $t$ count the degrees of the first and second
roots. Then, we have
\[
C^\bullet(x,w) = e^{ B^\bullet(x C'(x),w)}
\]
and
\begin{align*}
C^{\bullet\bullet}(x,w,t) &=
 \frac{x}{(xC'(x))'} \frac{\partial}{\partial x}C^\bullet(x,w)
\frac{\partial}{\partial x} C^\bullet(x,t) \\
& +  B^{\bullet\bullet}(x C'(x),w,t)   C^\bullet(x,w)C^\bullet(x,t).
\end{align*}
\end{lemma}

\section{Series-parallel graphs: asymptotics}\label{secSPAsymp}

We now analyze the singularities of the generating functions derived
in the previous section. In contrast to the case of outerplanar
graphs, the singularities of 2-connected and connected SP graphs are
of the same type (square-roots and powers of square-roots), so that
a single pair of technical lemmas is sufficient to deal with both
cases. As before, we derive asymptotic estimates for the numbers
$d_{n,k}$ and $d_{n,k,\ell}$ as $n,k,\ell \to \infty$ in a suitable
range, and we obtain a precise estimate for the maximum degree of
2-connected and connected SP graphs by applying the Master Theorem
discussed in the Introduction.

\subsection{Series-Parallel Networks}

We first discuss the singular behaviour of the functions $E(x)$ and
$D(x,w)$ associated to SP networks.

\bl[From {\cite[Lemma~2.2]{BGKN-series-parallel}}] \label{Le8.2}
The function $E(x)$ admits a local expansion of the form (\ref{eqyx})
or, equivalently, of the form
\[
E(x) = E_0 + E_1 X + E_2 X^2 + E_3 X^3\cdots,
\]
where $X=\sqrt{1 - x/\rho_1}$ and $\rho_1 \approx 0.1280038$. More
precisely, we have that $E_0$, $\rho_1$ and $E_1$ satisfy the
equations
\begin{align}
E_0+1&= 2 \exp\left( \frac 1{ 1 + 1/{E_0} + \sqrt{1 + 1/{E_0}} } \right),  \label{eqLe8.2}\\
\rho_1  &= \frac{\sqrt{1-1/E_0 } -1 }{E_0 }, \notag \\
E_1  &=  -\sqrt{\frac{2E_0(1+E_0)}{4+3\rho_1 E_0} }, \notag
\end{align}
from where we obtain
$E_0\approx 1.867893$ 
and $E_1\approx-1.507045$. 
\el

\begin{proof}
We start with the implicit equation (\ref{eqEx}).
By \cite[Theorem~2.19]{Drm-randomtrees} the generating function
$E(x)$  has a singular representation of the form
\begin{align*}
E(x) &= g(x) - h(x)\sqrt{1- \frac x{\rho_1}} \\
& = E_0 + E_1 X + E_2  X^2 +  E_3 X^3 + \cdots,
\end{align*}
where $E_0 = E(\rho_1)$ and $\rho_1$ are determined by the
system of equations
\begin{align*}
E_0 &= 2 \exp\left(\frac{\rho_1E_0^2}{1+\rho_1E_0} \right) - 1,\\
1 &= 2 \exp\left(\frac{\rho_1E_0^2}{1+\rho_1E_0} \right)
\frac{\rho_1 E_0( 2 + \rho_1 E_0)}{(1 + \rho_1 E_0)^2}.
\end{align*}
From here we obtain that
\begin{equation}\label{eqrhoE0}
\rho_1 E_0^2\left( 2 + \rho_1 E_0\right) = 1,
\end{equation}
and hence Equation~\ref{eqLe8.2} in  the statement, and the
approximations $\rho_1 \approx 0.128004$ and $E_0 \approx 1.867893$.
Furthermore,
\[
E_1 = -\sqrt{\frac{2E_0(1+E_0)}{4+3\rho_1 E_0} }
\approx -1.507045 \ne 0.
\]
\end{proof}

\bl \label{Le8.2bis}
The function $D(x,w)$ has a local expansion of the form
\begin{equation}\label{eqDrep}
D(x,w) = D_0(x)+ D_1(x) W + D_2(x) W^2 + \cdots,
\end{equation}
where $W=\sqrt{1 - w/w_0(x)}$ and
\begin{align*}
w_0(x) &= \left( 1 + \frac 1{xE(x)} \right)
\exp\left( - \frac 1{1+xE(x)} \right) -1,\\
D_0(x) & = \frac 1{xE(x)},\\
D_1(x) &= -\left( 1 + \frac 1{xE(x)} \right)
\sqrt{\frac{2 w_0(x) }{1+w_0(x)}}, \\
D_2(x) &= -\frac{2}{3}\left(\exp\left(\frac
1{1+xE(x)}\right)-1-\frac 1{xE(x)}\right).
\end{align*}
In particular, the functions $w_0(x)$, $D_0(x)$, $D_1(x)$ and
$D_2(x)$ have a singular expansion in $X$ analogous to $E(x)$. \el

\begin{proof}
Recall that the generating function $D(x,w)$ satisfies the
Equation~(\ref{eqDxw-imp}). We first consider $x$ as a (complex)
parameter and observe, by another application of
\cite[Theorem~2.19]{Drm-randomtrees}, that $D(x,w)$ has a
representation of the form
\begin{align}
D(x,w) &= g(x,w)- h(x,w) \sqrt{1- \frac w{w_0(x)}} \nonumber \\
&= D_0(x) + D_1(x)
W +
D_2(x)   W^2 +
D_3(x) W^2 + \cdots,   \label{eqDxw-exp}
\end{align}
where $W = \sqrt{ 1 - w/{w_0(x)}}$ and
where $w_0(x)$ and $D_0(x)$ are determined by the system
of equations
\begin{align*}
D_0(x) &= (1+w_0(x)) \exp\left(\frac{xD_0(x)E(x)}{1+xE(x)} \right) - 1,\\
1 &= (1+w_0(x)) \exp\left(\frac{xD_0(x)E(x)}{1+xE(x)} \right)
\frac {xE(x)}{1+xE(x)}.
\end{align*}
In particular we obtain the representations claimed for $w_0(x)$,
$D_0(x)$, $D_1(x)$ and $D_2(x)$.
\end{proof}

Note that representations similar to those of Lemma~\ref{Le8.2} and
\ref{Le8.2bis} hold for $S(x) = xE(x)^2/(1+xE(x))$ and $S(x,w) =
xE(x)D(x,w)/(1+xE(x))$, respectively. We also note for future use
that $w_0(x)$ of Lemma~\ref{Le8.2bis} satisfies $w_0(\rho_1) \approx
1.312267 > 1$.  

Finally, we remark that (\ref{eqDrep}) can be rewritten as
\begin{equation}\label{eqDrep2}
D(x,w) =  G(x,X,w) - H(x,X,w)\sqrt{1- y(x)w},
\end{equation}
where $X = \sqrt{1- x/\rho_1}$, $y(x) = 1/w_0(x)$, and $G(x,v,w)$
and $H(x,v,w)$ are analytic functions that are non-zero
for $(x,v,w) = (\rho_1,0,w_0(\rho_1))$.

\subsection{$2$-Connected Series-Parallel Graphs}
We first recall the asymptotic estimate for the number of
$2$-connected SP graphs. From \cite[Theorem
2.5]{BGKN-series-parallel}, the number of labelled $2$-connected SP
graphs is given asymptotically by
\[
b_n = b\cdot
 n^{-\frac 52} \rho_1^{-n} n! \left( 1+ O\left( \frac 1n\right)\right),
\]
where $\rho_1 \approx 0.1280038$ and $b \approx 0.0010131$.




Next we derive the asymptotic estimates for $d_{n,k}$ and
$d_{n,k,\ell}$ that we need in order to apply Theorem~\ref{Th1}.

\begin{prop}\label{Pro5}
Let $d_{n,k}$ denote the probability that a randomly selected
vertex in a $2$-connected SP graph with $n$ vertices
has degree $k$. Then we have uniformly for $k\le C \log n$
\begin{equation}
d_{n,k} \sim \overline d_k,
\end{equation}
where $\overline d_k$ denotes the asymptotic degree distribution of
$2$-connected SP graphs encoded by the generating function
\[
\overline p(w) =  \sum_{k\ge 2} \overline d_k w^k  =
\frac{B_1(w)}{B_1(1)},
\]
with
\begin{align*}
D_0( w) &= (1+w) \exp\left( {\frac {\rho_1 E_0 }{1+\rho_1 E_0 }D_0( w)}\right) - 1,\\
B_1( w) &= \frac{E_1 \rho_1 ^2D_0( w)^2}{2(1+\rho_1 E_0 )^2},
\end{align*}
and $\rho_1$, $E_0$, and $E_1$ are as in Lemma~\ref{Le8.2}. The
$\overline d_k$ are given asymptotically  by
\begin{equation}\label{eqdkSP2c}
\overline d_k = c\, w_0(\rho_1)^{-k} k^{-\frac 32}
\left(1 +  O\left( \frac 1k \right) \right),
\end{equation}
where $c>0$ is some constant.

Furthermore, we have uniformly for all $n,k\ge 1$
\begin{equation}
d_{n,k}= O(\overline q^ k)
\end{equation}
for some real $\overline q$ with $0< \overline q < 1$.
\end{prop}

The proof of Proposition~\ref{Pro5} makes use of the following
technical lemma, which we prove in Appendix~B.

\begin{lemma}\label{Le7}
Suppose that a generating function $f(x,w)= \sum_{n,k\ge 0}
f_{n,k} x^n w^k$ with non-negative coefficients $f_{n,k}$ has a
local representation of the form
\[
f(x,w) = G(x,X,w) + H(x,X,w) \left( 1 - y(x)w \right)^{\frac 32},
\]
where $X = \sqrt{1-x/x_0}$, the functions $G(x,v,w)$ and $H(x,v,w)$
are non-zero and analytic at $(x,v,w) = (x_0,0,1/g(x_0))$, and
$y(x)$ is a power series of the square-root type as in (\ref{eqyx}).

Then we have uniformly for $k \le C \log n$ (with an arbitrary
constant $C > 0$)
\[
f_{n,k} = \frac{3 h(x_0) H(x_0,0,1/g(x_0))}{8 \pi} g(x_0)^{k-1}
x_0^{-n} k^{-\frac 32} n^{-\frac 32} \left( 1 + O\left( \frac 1k
\right) \right).
\]
Moreover, for every $\varepsilon > 0$ we have uniformly for all
$n,k\ge 0$
\[
f_{n,k} = O\left( (g(x_0)+\varepsilon)^{k}   \rho^{-n} n^{-\frac
32} \right).
\]

If $g(x_0)<1$, then $f_n = \sum_{k} f_{n,k}$ is given asymptotically
by
\begin{align*}
f_n &= \frac{1}{2\sqrt{\pi}}
\left( G_v(x_0,0,1) + (1-g(x_0))^{3/2}\left(H_v(x_0,0,1) - \frac{3 h(x_0)}{2 (1-g(x_0))} H(x_0,0,1) \right) \right) \\
&\qquad \times x_0^{-n}  n^{-\frac 32} \left( 1 + O\left( \frac 1n
\right)\right),
\end{align*}
and for every $k$ the limit
\[
\overline d_k = \lim_{n\to\infty} \frac{f_{n,k}}{f_n} \] exists.
\end{lemma}

\begin{proof}[Proof of Proposition \ref{Pro5}]
It is not difficult to check that $B^\bullet(x,w)$ fits precisely
into the assumptions of Lemma~\ref{Le7}. By using
Lemma~\ref{Lesp22} and Equation~(\ref{eqDxw-exp}) it follows that
$B^\bullet(x,w)$ has a singular representation of the form
\begin{equation}\label{eqLe101}
B^\bullet(x,w) = B_0^\bullet(x) + B_2^\bullet(x) W^2 +
B_3^\bullet(x) W^3  + \cdots,
\end{equation}
where $W = \sqrt{1- w/{w_0(x)}}$ and the coefficients
$B_j^\bullet(x)$ are analytic functions in $x$ and $E(x)$. For
example, we have
\begin{align*}
B_0^\bullet(x) &= \frac 1{2E(x)(1+xE(x))}, \\
B_2^\bullet(x) &= -\frac {x^2E(x)D_1(x)^2}{2(1+xE(x))}, \\
B_3^\bullet(x) &= -\frac {x^2E(x)D_1(x)D_2(x)}{1+xE(x)}.
\end{align*}
Clearly, this representation can be rewritten as
\begin{equation}\label{eqBbulletrep}
B^\bullet(x,w) = G(x,X,w) + H(x,X,w)\left( 1 - y(x)
w\right)^{3/2},
\end{equation}
where $X = \sqrt{1 - x/\rho_1}$ and $y(x) = 1/w_0(x)$.

There is an alternative way to derive the same result without making
use of the explicit representation for $B^\bullet(x,w)$.  Start from
(\ref{eqDrep2}) to derive a corresponding representation for
\[
\frac{\partial}{\partial w} B^\bullet (x,w) = \sum_{k\ge 1} k B_k(x)
w^k = x e^{S(x,w)} = \widetilde G(x,X,w) - \widetilde H(x,X,w)\left(
1 - y(x) w\right)^{1/2}.
\]
By expanding the functions $\widetilde G(x,X,w)$ and $\widetilde
H(x,X,w)$ in $W = \sqrt{1-y(x)w}$ this leads to a local
representation of the form
\[
\frac{\partial}{\partial w} B^\bullet (x,w) = G_0(x,X) + G_1(x,X)W
+ G_2(x,X)W^2 + \cdots.
\]
Finally, since
\[
\int W^\ell\, dw = - \frac 2{(\ell+2) y(x)} W^{\ell+1} + C,
\]
we obtain a representation for $B^\bullet (x,w)$ of the form
\begin{equation}\label{eqBbulletrep2}
B^\bullet (x,w) = \widetilde G_0(x,X) +  \widetilde G_2(x,X)W^2 +
\widetilde G_3(x,X)W^3 + \cdots,
\end{equation}
where
\[
\widetilde G_0(x,X) = \int_0^{1/y(x)} \frac{\partial}{\partial w}
B^\bullet (x,w)\, dw
\]
and
\[
\widetilde G_j(x,X) = - \frac{2G_{j-2}(x,X)}{jy(x)}
\]
for $j\ge 2$. Of course, (\ref{eqBbulletrep2}) rewrites to
(\ref{eqBbulletrep}).

Anyway, this shows that $B^\bullet (x,w)$ has precisely the form of
$f(x,w)$ in Lemma~\ref{Le7}, and that $w(\rho_1)>1$ implies that
$y(\rho_1)=g(\rho_1)<1$. Hence, all the properties claimed follow.
The representation for $\overline p(w)$ and the asymptotic expansion
of $\overline d_k$ can be also found in \cite{DGN2}.
\end{proof}

\begin{prop}\label{Pro6}
Let $d_{n,k,\ell}$ denote the probability that two different (and
ordered) randomly selected vertices in a $2$-connected
SP graph with $n$ vertices have degrees $k$ and
$\ell$. Then we have uniformly for $2\le k,\ell \le C \log n$
\begin{equation}
d_{n,k,\ell} \sim \overline d_k \overline d_\ell.
\end{equation}
Furthermore, we have uniformly for all $n,k\ge 1$
\begin{equation}
d_{n,k,\ell}= O(\overline q^ {k+\ell})
\end{equation}
for some real number $\overline q$ with $0< \overline q < 1$.
\end{prop}

The proof of the proposition makes use of the following lemma.

\begin{lemma}\label{Le8}
Let $f(x,w,t) = \sum_{n,k,\ell} f_{n,k,\ell} x^n w^k t^\ell$ be a
generating function of non-negative numbers $f_{n,k,\ell}$, and
suppose that $f(x,w,t)$ can be represented as
\begin{align}\label{eqLe81}
&f(x,w,t) = \frac 1X \Bigl( G_1(x,X,w,t) + G_2(x,X,w,t) \left( 1- y(x)w \right)^{1/2} \\
&\qquad+ G_3(x,X,w,t) \left( 1- y(x)t \right)^{1/2} +
G_4(x,X,w,t)\left( 1- y(x)w \right)^{1/2} \left( 1- y(x)t
\right)^{1/2} \Bigr), \nonumber
\end{align}
where $X = \sqrt{1-x/x_0}$, the functions $G_j(x,v,w)$ are analytic
for at $(x,v,w,t) = (x_0,0,1/g(x_0),1/g(x_0))$ for $j= 1,2,3,4$, and
non-zero for $j=4$, and $y(x)$ is a power series of the square-root
type as in (\ref{eqyx}).

Then, we have uniformly for $k,\ell \le C \log n$ (with an
arbitrary constant $C > 0$)
\[
f_{n,k,\ell} =  \frac{G_4\left(x_0,0,\frac 1{g(x_0)}, \frac
1{g(x_0)}\right)}{4 \pi^{3/2}} g(x_0)^{k+\ell} x_0^{-n}
(k\ell)^{-\frac 32} n^{-\frac 12} \left( 1 + O\left( \frac 1{\sqrt
k} + \frac 1{\sqrt \ell} \right) \right).
\]
Moreover, for every $\varepsilon > 0$ we have uniformly for all
$n,k\ge 0$
\[
f_{n,k,\ell} = O\left( (g(x_0)+\varepsilon)^{k+\ell}  x_0^{-n}
n^{-\frac 12} \right).
\]

If $g(x_0)<1$, then $f_n = \sum_{k,\ell} f_{n,k,\ell}$ is given
asymptotically  by
\begin{align*}
f_n &= \frac 1{\sqrt{\pi}} \left( G_1 + \left( G_2+G_3 \right)
\sqrt{1-g(x_0)}
+ G_4 (1-g(x_0))\right) \\
&\qquad \times x_0^{-n}  n^{-\frac 12} \left( 1 + O\left( \frac 1n
\right)\right),
\end{align*}
in which $G_j$ ($j= 1,2,3,4$) has to be evaluated at $(x,v,w,t) =
\left(x_0,1,\frac 1{g(x_0)},\frac 1{g(x_0)}\right)$, and for every
pair $(k,\ell)$ the limit
\[
d_{k,\ell} = \lim_{n\to\infty} \frac{f_{n,k,\ell}}{f_n} \] exists.
\end{lemma}

\begin{proof}[Proof of Proposition \ref{Pro6}]
Since we do not have an explicit expression for
$B^{\bullet\bullet}(x,w,t)$, we proceed using the second idea in the
proof of Proposition~\ref{Pro5}.  We start with the explicit
expression for $\frac{\partial}{\partial w}B^{\bullet \bullet}(x,w)$
from Lemma~\ref{Lesp22}. By using the local singular expansion
(\ref{eqDrep2}) for $D(x,w)$,  it follows directly that
\[
D(x,w)D(x,t)
\left( (1+wt)\exp\left( \frac{x}{1+xE(x)} D(x,w)D(x,t) \right) -1 \right)\\
\]
and that $D(x,w)^2D(x,t)^2$ can be represented as
\[
G_1(x,X,w,t) + G_2(x,X,w,t)W + G_3(x,X,w,t)T + G_4(x,X,w,t)WT,
\]
where $T = \sqrt{1-t/w_0(x)}$, and $w$ and $t$ vary in a
$\Delta$-domain corresponding to the common singularity $w_0(x)
=1/y(x)$. Furthermore, since $D_0(x) = 1/(xE(x))$ we have
\[
\frac 1{1-xE(x)D(x,w)} = -\frac{D_0(x)}{D_1(x) W}
\left( 1- \frac{D_2(x)}{D_1(x)}W + O(W^2) \right).
\]
Finally, since $\rho_1 E_0^2\left( 2 + \rho_1 E_0\right) = 1$
(see (\ref{eqrhoE0})) we have the expansion
\[
\frac{1}{1-2xE(x)^2-x^2E(x)^3} =
\frac{F_{-1}}X + F_0 + F_1 X + O(X^2),
\]
where $X = \sqrt{1 - x/\rho_1}$.

Consequently, the function $\frac {\partial}{\partial w}
B^{\bullet\bullet}(x,w,t)$ can be represented as
\[
\frac 1{XW} \left( H_1(x,X,w,t) + H_2(x,X,w,t)W + H_3(x,X,w,t)T +
H_4(x,X,w,t)WT \right),
\]
with analytic functions $H_j(x,v,w,t)$. Hence, by rewriting this
representation as a series in $W$, integration with respect to
$w$ leads, as in the proof of Proposition~\ref{Pro5}, to a representation of
$B^{\bullet\bullet}(x,w,t)$ of the form
\begin{align*}
&B^{\bullet\bullet}(x,w,t) \\
&= \frac 1X \left(\widetilde H_1(x,X,w,t) + \widetilde H_2(x,X,w,t)W
+ \widetilde H_3(x,X,w,t)T + \widetilde H_4(x,X,w,t)WT   \right).
\end{align*}
Hence, we can apply Lemma~\ref{Le8} to obtain all the claimed
properties.
\end{proof}

We are ready to prove the main result in this section.

\begin{theo}\label{Th2csp}
Let $\Delta_n$ denote the maximum degree of a random labelled
$2$-connected SP graph with $n$ vertices.
Then
\[
\frac{\Delta_n}{\log n} \to c \qquad\mbox{in probability}
\]
and
\[
\mathbb{E}\, \Delta_n \sim  c \log n
\qquad (n\to\infty),
\]
where $c = \displaystyle\frac 1{\log (1/q)}\approx 3.679772$, and
$q\approx 0.7620402$ is given by
\[ q = \left( \left(1+\frac 1{\rho_1E_0})\right) \exp\left(-\frac{1}{1+\rho_1E_0}\right)-1 \right)^{-1}, \]
with $\rho_1$ and $E_0$ as in Lemma~\ref{Le8.2}.
\end{theo}

\begin{proof}[Proof of Theorem~\ref{Th2csp}]
The proof is a direct application of Propositions~\ref{Pro5},
\ref{Pro6} and Theorem~\ref{Th1}.
\end{proof}

\begin{rem}
It is also possible to prove in this case that $\overline p(w,t) =
\overline p(w) \overline p(t)$.  However, it is much more technical
than in the case of $2$-connected outerplanar graphs.  For the sake
of conciseness we skip the details.
\end{rem}

\subsection{Connected Series-Parallel Graphs}

We first analyze the equation for $C'(x)$. From \cite[Theorem
3.7]{BGKN-series-parallel}, the number $c_n$ of labelled connected
SP graphs is given asymptotically by
\[
c_n = c\cdot
 n^{-\frac 52} \rho_2^{-n} n! \left( 1+ O\left( \frac 1n\right)\right),
\]
where $\rho_2 \approx 0.11021$ and $c \approx 0.0067912$.


The function
$v(x) = xC'(x)$ has a local representation of the form
\begin{equation}\label{eqxCdashx}
x C'(x) = g(x) - h(x) \sqrt{1- \frac x{\rho_2}}.
\end{equation}
Thus, from an analytic point of view, we are in the same situation as
in the outerplanar case. Recall that
\[ C^\bullet(x,w) = e^{ B^\bullet(x C'(x),w)}.\]
The singularity $\rho_1$ of $B'(x)$ has no influence on the singular
behavior of $C'(x)$, since it is not hard to check that
$v(\rho_2)=\rho_2C'(\rho_2)<\rho_1$.
Consequently we obtain
corresponding representations for
\[
C(x) = g_2(x) + h_2(x) \left( 1- \frac x{\rho_2}\right)^{\frac 32}
\]
and the asymptotic expansion for $c_n$.

The final step is to prove the following properties.
\begin{prop}\label{Pro7}
Let $d_{n,k}$ denote the probability that a randomly selected
vertex in a connected SP graph with $n$ vertices
has degree $k$. Then we have uniformly for $k\le C \log n$
\begin{equation}
d_{n,k} \sim \overline d_k,
\end{equation}
where $\overline d_k$ denotes the asymptotic degree distribution of
connected SP graphs encoded by the generating function
\[
\overline p(w) =  \sum_{k\ge 2} \overline d_k w^k  = \rho_2 e^{ B^\bullet(v_0,w)}
\frac{\partial}{\partial x} B^\bullet(v_0,w).
\]
where $v_0=v(\rho_2)$, and is given asymptotically  by
\begin{equation}\label{eqcpwSPc}
\overline d_k = c'\cdot k^{-\frac 32} w_0(v_0)^{-k}
\left( 1 + O\left( \frac 1k \right) \right),
\end{equation}
where $c' \approx 3.5952391$.

Furthermore, we have uniformly for all $n,k\ge 1$
\begin{equation}
d_{n,k}= O(\overline q^ k)
\end{equation}
for some real $\overline q$ with $0< \overline q < 1$.
\end{prop}

\begin{proof}
By using (\ref{eqLe101}) we derive the singular representation
\[
e^{B^\bullet(x,w)} = e^{B_0^\bullet(x)}
\left( 1 + B_2^\bullet(x) W^2 + B_3^\bullet(x) W^3 + O(W^4) \right).
\]
By using this local expansion and (\ref{eqxCdashx}) we get
\begin{align}
C^\bullet(x,w) &= e^{B^\bullet(xC'(x),w)} \nonumber \\
&= C_0^\bullet(x) + C_2^\bullet(x) \overline W^2 + C_3^\bullet(x)
\overline W^3 + \cdots,
\label{eqCbulletexpSP} \\
&= G(x,X_2,w) + H(x,X_2,w)\left( 1- \overline y(x) w \right)^{3/2},
\nonumber
\end{align}
where all functions $C_j^\bullet(x)$ have a square-root singularity
of the form (\ref{eqyx}) with $x_0 = \rho_2$, $X_2 =
\sqrt{1-x/\rho_2}$, and with $\overline y(x) = 1/w_0(xC'(x))$. We
have $\overline W = \sqrt{1- \overline y(x) w}$, where the function
$\overline y(x) = 1/w_0(xC'(x)) = \overline g(x) - \overline h(x)
X_2$ has also a square-root singularity of the form (\ref{eqyx})
with $x_0 = \rho_2$.

Hence we are exactly in the same situation as in the $2$-connected
case and we can apply Lemma~\ref{Le7}.
\end{proof}

\begin{prop}\label{Pro8}
Let $d_{n,k,\ell}$ denote the probability that two different
(and ordered) randomly selected
vertices in a connected SP graph with $n$ vertices
have degrees $k$ and $\ell$.
Then we have uniformly for $2\le k,\ell \le C \log n$
\begin{equation}
d_{n,k,\ell} \sim \overline d_k \overline d_\ell.
\end{equation}
Furthermore, we have uniformly for all $n,k\ge 1$
\begin{equation}
d_{n,k,\ell}= O(\overline q^ {k+\ell})
\end{equation}
for some real number $\overline q$ with $0< \overline q < 1$.
\end{prop}

\begin{proof}
We consider the function $C^{\bullet\bullet}(x,w,t)$.
We focus first on the term
\[
 \frac{x}{(xC'(x))'} \frac{\partial}{\partial x}C^\bullet(x,w)
\frac{\partial}{\partial x} C^\bullet(x,t).
\]
Since
\begin{align*}
\frac{\partial}{\partial x} X_2 &= -\frac 1{2x_2 X_2},\\
\frac{\partial}{\partial x} \overline y(x) &= \frac 1{X_2}\left(
\frac 1{2x_2}\overline h(x) - \overline h'(x)X_2^2 + \overline g'(x)
X_2 \right),
\end{align*}
it follows from (\ref{eqCbulletexpSP}) that
$\frac{\partial}{\partial x}C^\bullet(x,w)$ can be represented
as
\[
\frac{\partial}{\partial x}C^\bullet(x,w) = \frac 1{X_2}\left(
\overline G(x,X_2,w) - \overline H(x,X_2,w) \overline W \right).
\]
Hence, we obtain
\begin{align*}
& \frac{x}{(xC'(x))'} \frac{\partial}{\partial x}C^\bullet(x,w)
\frac{\partial}{\partial x} C^\bullet(x,t) \\
&= \frac 1{X_2} \left( H_1(x,X_2,w,t) + H_2(x,X_2,w,t)\overline W +
H_3(x,X_2,w,t)\overline T + H_4(x,X_2,w,t)\overline W\overline T
\right),
\end{align*}
for certain analytic functions $H_j$.

Since $v_0 = \rho_2 C'(\rho_2) < \rho_1$, the function $X(xC'(x)) =
\sqrt{1-xC'(x)/\rho_1}$ is analytic at $x = \rho_2$. Consequently,
the second term
\[
B^{\bullet\bullet}(x C'(x),w,t)   C^\bullet(x,w)C^\bullet(x,t)
\]
can be
represented as
\[
J_1(x,X_2,w,t) + J_2(x,X_2,w,t)\overline W + J_3(x,X_2,w,t)\overline
T + J_4(x,X_2,w,t)\overline W\,\overline T,
\]
with analytic functions $J_j$. Hence we obtain
\begin{align*}
&C^{\bullet\bullet}(x,w,t) \\
&= \frac 1{X_2} \left(\widetilde H_1(x,X_2,w,t) + \widetilde
H_2(x,X_2,w,t)\overline W + \widetilde H_3(x,X_2,w,t)\overline T +
\widetilde H_4(x,X_2,w,t) \overline W\, \overline T   \right),
\end{align*}
with $\widetilde H_j = H_j + X_2 J_j$.

Thus, we can apply Lemma~\ref{Le8} and obtain (as in the 2-connected
case) all the properties claimed.
\end{proof}

\begin{theo}\label{Thcsp}
Let $\Delta_n$ denote the maximum degree of a random
labelled connected SP graph with $n$ vertices.
Then
\[
\frac{\Delta_n}{\log n} \to c \qquad\mbox{in probability}
\]
and
\[
\mathbb{E}\, \Delta_n \sim  c \log n
\qquad (n\to\infty),
\]
where $c = \displaystyle\frac 1{\log(1/q)}\approx 3.482774$, and
$q\approx 0.750416$ is given by
\[ q = \left( \left(1+\frac 1{\tau E(\tau)}\right) \exp\left(-\frac{1}{\tau E(\tau)}\right)-1 \right)^{-1}, \]
with $\tau = \rho_1C'(\rho_1)$.
\end{theo}

\begin{proof}
Again, the proof is a direct consequence of Proposition~\ref{Pro7},
\ref{Pro8} and Theorem~\ref{Th1}.
\end{proof}

\begin{rem}
We recall that the radius of convergence of $C'(x)$ is smaller than
that of $B'(x)$. Consequently,  for the class of connected SP graphs
we obtain automatically that $\overline p(w,t) = \overline p(w)
\overline p(t)$.
\end{rem}

\newpage
\section*{Appendix A}

The proof of Theorem~\ref{Th1} is based on the so-called
{\it first and second moment methods}.

\bl\label{Le1}
Let $X$ be a discrete random variable on non-negative integers
with finite first moment.
Then
\[
\mathbb{P}\{ X > 0\} \le \min\{ 1, \mathbb{E}\, X \} .
\]
Furthermore, if $X$ is a non-negative random variable which is not
identically zero and has finite second moment then
\[
\mathbb{P}\{ X > 0\} \ge \frac{(\mathbb{E}\, X)^2}{\mathbb{E}\, (X^2)}.
\]
\el

\bpf
For the first inequality we only have to observe that
\[
 \mathbb{E}\, X = \sum_{k\ge 0} k\, \mathbb{P}\{X = k\} \ge \sum_{k\ge 1} \mathbb{P}\{X = k\} = \mathbb{P}\{X > 0\}.
\]
The second inequality follows
from an application of the Cauchy-Schwarz inequality:
\[
\mathbb{E}\, X = \mathbb{E}\, \left(X \cdot {\bf 1}_{[X> 0]}\right) \le
\sqrt{\mathbb{E}\, (X^2)}\sqrt{ \mathbb{E}\, ({\bf 1}_{[X> 0]}^2  ) } =
\sqrt{\mathbb{E} (X^2)}\sqrt{ \mathbb{P}\{ X > 0\}  }.
\]
\epf

As indicated in the Introduction, we apply this principle for the
random variable $Y_{n,k}$ that counts the number of vertices of
degree $>k$ in a random graph with $n$ vertices. This random
variable is closely related to the maximum degree $\Delta_n$ by the
relation
\[
Y_{n,k} > 0 \quad \Longleftrightarrow \quad \Delta_n> k.
\]
One of our aims is to get bounds for the expected maximum degree
$\mathbb{E}\, \Delta_n $. Due to the relation
\begin{align*}
\mathbb{E}\, \Delta_n &= \sum_{k \ge 0} \mathbb{P}\{\Delta_n > k\} \\
&=\sum_{k \ge 0} \mathbb{P}\{Y_{n,k} > 0\},
\end{align*}
we are actually led to estimate the probabilities
$\mathbb{P}\{Y_{n,k} > 0\}$, which can be handled via the first and
second moment methods by estimating the first two moments
\[
\mathbb{E}\, Y_{n,k} \quad\mbox{and}\quad \mathbb{E}\, Y_{n,k}^2.
\]

Actually, we compute asymptotics for the
probabilities $d_{n,k}$. Observe that
the number $X_{n,k}$ of vertices of degree $k$ is given by
\begin{equation}\label{eqXnkrel}
X_{n,k} = \sum_{v\in V(G_n)} {\bf 1}_{[d(v)= k]},
\end{equation}
and consequently we have
\[
\mathbb{E}\, X_{n,k} = n d_{n,k}.
\]
Since $Y_{n,k} = \sum_{\ell > k} X_{n,\ell}$, we have
\begin{equation}\label{eqEYnkrep}
\mathbb{E}\, Y_{n,k} = n \sum_{\ell > k} d_{n,\ell}.
\end{equation}

The second moment is a bit more involved. From (\ref{eqXnkrel})
we get
\begin{align*}
X_{n,k}^2 &= \sum_{v,w\in V(G_n)} {\bf 1}_{[d(v)= k]}  {\bf 1}_{[d(w)= k]}  \\
&= \sum_{v\in V(G_n)} {\bf 1}_{[d(v)= k]} + \sum_{v,w\in V(G_n),
v\ne w } {\bf 1}_{[d(v)= k \, \wedge\, d(w)= k]},
\end{align*}
and consequently
\[
\mathbb{E}\, X_{n,k}^2 = n d_{n,k} + n(n-1) d_{n,k,k},
\]
where $d_{n,k,k}$ denotes the probability that two different
randomly selected vertices have degree $k$. Similarly,  for $k\ne
\ell$ we have
\begin{align*}
X_{n,k}X_{n,\ell} &=
\sum_{v,w\in V(G_n)} {\bf 1}_{[d(v)= k]}  {\bf 1}_{[d(w)= \ell]}  \\
&= \sum_{v,w\in V(G_n), v\ne w } {\bf 1}_{[d(v)= k \, \wedge\, d(w)= \ell]}
\end{align*}
and
\[
\mathbb{E}\, X_{n,k}X_{n,\ell} =  n(n-1) d_{n,k,\ell},
\]
where $d_{n,k,\ell}$ denotes the probability that two different randomly
selected vertices have degrees $k$ and $\ell$.

This also shows that
\begin{align*}
\mathbb{E}\, Y_{n,k}^2 &= \sum_{\ell_1,\ell_2 > k} \mathbb{E}\,
X_{n,\ell_1}X_{n,\ell_2} \\
&= \sum_{\ell > k} \mathbb{E}\, X_{n,\ell}^2 +
2 \sum_{\ell_1 > \ell_2 > k} \mathbb{E}\,X_{n,\ell_1}X_{n,\ell_2}  \\
&= n \sum_{\ell >k} d_{n,\ell} + n(n-1) \sum_{\ell > k} d_{n,\ell,\ell} \\
&+ 2 n(n-1) \sum_{\ell_1 > \ell_2 > k} d_{n,\ell_1,\ell_1} \\
&=  n \sum_{\ell >k} d_{n,\ell} + n(n-1) \sum_{\ell_1,\ell_2 > k}d_{n,\ell_1,\ell_1}.
\end{align*}

Summing up we have the following estimates. \bl\label{Le2} Suppose
that $d_{n,k}$ denotes the probability that a randomly selected
vertex in a graph of size $n$ from  a certain class of random planar
graphs has degree $k$, and that $d_{n,k,\ell}$ denotes the
probability that two randomly selected (ordered) vertices have
degrees $k$ and $\ell$. Furthermore let $\Delta_n$ denote the
maximum degree of a random planar graph (in this class) of size $n$.

Then the probability $\mathbb{P}\{ \Delta_n > k\}$ is bounded by
\begin{equation}\label{eqL20}
\frac{ n^2 \left( \sum_{\ell > k} d_{n,\ell} \right)^2 }
{ n \sum_{\ell >k} d_{n,\ell} + n(n-1) \sum_{\ell_1,\ell_2 > k}d_{n,\ell_1,\ell_1}}
\le \mathbb{P}\{ \Delta_n > k\}
\le \min\left\{ 1,
 n \sum_{\ell > k} d_{n,\ell} \right\}.
\end{equation}
Consequently the expected value satisfies
\begin{equation}\label{eqL21}
\mathbb{E}\,\Delta_n \le \sum_{k\ge 0} \min\left\{ 1,
 n \sum_{\ell > k} d_{n,\ell} \right\}
\end{equation}
and
\begin{equation}\label{eqL22}
\mathbb{E}\,\Delta_n \ge \sum_{k\ge 0}
\frac{ n^2 \left( \sum_{\ell > k} d_{n,\ell} \right)^2 }
{ n \sum_{\ell >k} d_{n,\ell} + n(n-1) \sum_{\ell_1,\ell_2 > k}d_{n,\ell_1,\ell_1}}.
\end{equation}
\el

With the help of Lemma~\ref{eqL20} is it easy to
prove Theorem~\ref{Th1}.

\begin{proof}[Proof of Theorem~\ref{Th1}]
We start with the proof of \eqref{eqDeltanresult}. Suppose that the
constant $C$ of condition (2) satisfies $C > 2\max\left\{ (\log
(1/q))^{-1}, (\log (1/\overline q))^{-1} \right\}$, and define
$k_0(n)$ as
\[
k_0(n) = \min\left\{ k\ge 0: n \sum_{\ell > k} \overline d_\ell \le 1 \right\}.
\]
By assumptions (\ref{eqass1}) and (\ref{eqass3})  it follows that
\[
k_0(n) \sim \frac{\log n}{\log (1/q)} \qquad (n\to \infty).
\]
In particular we obtain from (\ref{eqL21})  that
\begin{align*}
\mathbb{E}\, \Delta_n &\le  \sum_{k\ge 0} \min\left\{ 1,
 n \sum_{\ell > k} d_{n,\ell} \right\} \\
&\le k_0(n) + 1 + n \sum_{\ell > k} d_{\ell} \, (1+o(1))    \\
&\sim  \frac{\log n}{\log q^{-1}}.
\end{align*}

Next define $k_1(n)$ by
\[
k_1(n) = \max\left\{ k\ge 0: n \sum_{\ell > k} \overline d_\ell \ge
\log n \right\},
\]
which also satisfies
\[
k_1(n) \sim \frac{\log n}{\log q^{-1}} \qquad (n\to \infty).
\]
By assumptions (\ref{eqass1})--(\ref{eqass3}) it  follows that,
uniformly for $0\le k \le k_1(n)$,
\[
n(n-1) \sum_{\ell_1,\ell_2 > k}d_{n,\ell_1,\ell_1}\sim
 n^2 \left( \sum_{\ell > k} d_{n,\ell} \right)^2
\]
and
\[
 n \sum_{\ell >k} d_{n,\ell} = o\left(
 n^2 \left( \sum_{\ell > k} d_{n,\ell} \right)^2\right).
\]
Consequently
\[
 \frac{ n \sum_{\ell >k} d_{n,\ell} + n(n-1) \sum_{\ell_1,\ell_2 > k}d_{n,\ell_1,\ell_1}-
 n^2 \left( \sum_{\ell > k} d_{n,\ell} \right)^2 }
{ n \sum_{\ell >k} d_{n,\ell} + n(n-1) \sum_{\ell_1,\ell_2 >
k}d_{n,\ell_1,\ell_1}} \to 0,
\]
uniformly for $0\le k \le k_1(n)$ as $n\to\infty$.

In order to obtain an upper bound for $\mathbb{E}\, \Delta_n$ we use
(\ref{eqL22}) and get, as $n\to\infty$,
\begin{align*}
\mathbb{E}\,\Delta_n &\ge \sum_{0\le k\le k_1(n)}
\frac{ n^2 \left( \sum_{\ell > k} d_{n,\ell} \right)^2 }
{ n \sum_{\ell >k} d_{n,\ell} + n(n-1) \sum_{\ell_1,\ell_2 > k}d_{n,\ell_1,\ell_1}} \\
&= \sum_{0\le k\le k_1(n)} 1 \\
& -\sum_{0\le k\le k_1(n)}
 \frac{ n \sum_{\ell >k} d_{n,\ell} + n(n-1) \sum_{\ell_1,\ell_2 > k}d_{n,\ell_1,\ell_1}-
 n^2 \left( \sum_{\ell > k} d_{n,\ell} \right)^2 }
{ n \sum_{\ell >k} d_{n,\ell} + n(n-1) \sum_{\ell_1,\ell_2 > k}d_{n,\ell_1,\ell_1}} \\
&\sim \frac{\log n}{\log q^{-1}}.
\end{align*}

Finally, it follows directly from assumptions \eqref{eqass1}--\eqref{eqass3}
and the estimate \eqref{eqL20} that for every $\varepsilon> 0$
\[
\mathbb{P}\left\{ \left| \frac{\Delta_n}{\log n} - \frac 1{\log q^{-1}} \right|
\ge \varepsilon \right\} \to 0
\]
as $n\to \infty$. This implies \eqref{eqDeltanresult0} and completes
the proof of Theorem~\ref{Th1}
\end{proof}

\begin{rem}
It is clear that asymptotic relations with error terms
(that are more precise than (\ref{eqass1})--(\ref{eqass3}))
imply an error term for the expected maximum degree
$\mathbb{E}\, \Delta_n$.
\end{rem}

\section*{Appendix B}

\begin{proof}[Proof of Lemma \ref{Le3}]
We use Cauchy's formula
\[
f_{n,k} = \frac{1}{(2\pi i)^2} \int_{\gamma} \int_{\Gamma}
\frac{ f(x,w) } {x^{n+1} w^{k+1}}\, dx\, dw
\]
with the following contours of integration.

For the integration with respect to $x$ we use
$\gamma = \gamma_1 \cup
\gamma_2 \cup \gamma_3 \cup \gamma_4$, where
\begin{eqnarray*}
\gamma_1 &=& \left\{ x = x_0\left(1 + \frac {-i+(\log n)^2  -t}n\right)
 : 0\le t \le (\log n)^2
\right\},\\
\gamma_2 &=& \left\{ x = x_0\left(1 - \frac 1n e^{-i\phi}\right):
-\frac \pi 2 \le \phi \le \frac \pi 2 \right\}, \\
\gamma_3 &=& \left\{ x = x_0\left(1 + \frac {i+t}n \right)
 : 0\le t \le (\log n)^2 \right\},
\end{eqnarray*}
and $\gamma_4$ is a circular arc centered at the origin and making
$\gamma$ a closed curve.

Similarly, for the integration with respect to $w$ we use
$\Gamma = \Gamma_1 \cup
\Gamma_2 \cup \Gamma_3 \cup \Gamma_4$, where
\begin{eqnarray*}
\Gamma_1 &=& \left\{ w = w_0\left( 1 + \frac {-i+(\log k)^2  -s}k
\right): 0\le s \le (\log k)^2
\right\},\\
\Gamma_2 &=& \left\{ w = w_0\left(1 - \frac 1k e^{-i\psi}\right):
-\frac \pi 2 \le \psi \le \frac \pi 2 \right\}, \\
\Gamma_3 &=& \left\{ w = w_0\left(1 + \frac {i+s}w \right)
: 0\le s \le (\log k)^2 \right\},
\end{eqnarray*}
where $w_0 = 1/g(x_0)$ and $\Gamma_4$ is a circular arc centered
at the origin and making $\Gamma$ a closed curve.

We recall that we assume that $k\le C\log n$ (for some constant
$C>0$). The following calculations will show that the Cauchy
integral is always well defined, in particular we have $1 - y(x)w
\ne 0$ for $x\in \gamma$ and $w\in \Gamma$. (Recall that $y(x)$ has
non-negative coefficients.) The most important part of the integral
comes from the $x\in \gamma_1 \cup \gamma_2 \cup \gamma_3$ and $w\in
\Gamma_1 \cup \Gamma_2 \cup \Gamma_3$.  If we use the substitutions
$x = x_0\left(1 + \frac tn\right)$ and $w = w_0\left(1 + \frac sk
\right)$, then $t$ and $s$ vary in a corresponding curve $H_1\cup
H_2 \cup H_3$ that can be considered as a finite part of a so-called
Hankel contour  $H$ (see Figure~\ref{f:hankelcontour}). In
particular we note that $|X| \le (\log n)/\sqrt n$ and $|w-w_0| \ge
w_0/k \ge w_0/(C \log n)$ on these parts of the integration.

\begin{figure}[htp] \centering
\includegraphics[width=0.5\hsize]{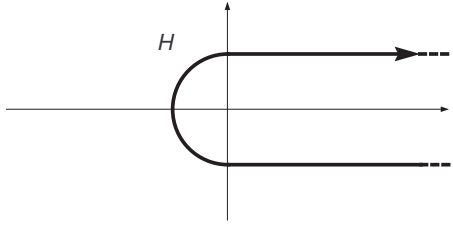}
\caption{Hankel contour of integration}
          \label{f:hankelcontour}
\end{figure}

By using the substitution $y(x) = g(x_0) - h(x_0)X + O(X^2)$ and the
relations $w_0g(x_0)=1$, we have
\[
1-y(x)w = g(x_0)(w_0-w) + h(x_0)w_0 X + O(X^2).
\]
Hence, if $n$ is large enough we definitely have $1-y(x)w\ne 0$, in
particular the term $h(x_0)w_0 X + O(X^2)$ is of minor order of
magnitude on these parts of the integration. Consequently we can
rewrite the reciprocal of $1-y(x)w$ as
\begin{align*}
\frac 1{1-y(x)w} &= \frac 1{g(x_0)(w_0-w) + h(x_0)w_0 X + O(X^2)} \\
&= \frac 1{g(x_0)(w_0-w)}
\left( 1 - \frac{h(x_0)w_0 X}{g(x_0)(w_0-w)} +
O\left( \frac {|X|^2}{ |w_0-w|^2} \right) \right) \\
&= \frac 1{1- \frac w{w_0}} - \frac{h(x_0)w_0 X}{\left( 1- \frac w{w_0}\right)^2}
+ O\left( \frac{|X|^2}{|w-w_0|^3} \right).
\end{align*}
Next we set $x = x_0(1-X^2)$ and rewrite $G(x,X,w)$ locally as
\[
G(x,X,w) = G_0(w) + G_1(w)X + O(X^2).
\]
Hence, if $x$ and $w$ are in this range then $f(x,w)$ can be represented as
\begin{align*}
f(x,w) &= \left( G_0(w) + G_1(w)X + O(X^2) \right)
\left( \frac 1{1- \frac w{w_0}} - \frac{h(x_0)w_0 X}{\left( 1- \frac w{w_0}\right)^2}
+ O\left( \frac{|X|^2}{|w-w_0|^3} \right) \right) \\
&= \frac{ G_0(w)}{1- \frac w{w_0}} -
\frac{G(x_0,0,w_0)h(x_0)w_0 X }{\left( 1- \frac w{w_0}\right)^2}
+ O\left( \frac{|X|}{\left| 1- \frac w{w_0} \right| } \right).
\end{align*}
The first term does not depend on $x$, hence it does not contribute
for $n\ge 1$. The second term provides the asymptotic leading term
\begin{align*}
&\frac{1}{(2\pi i)^2} \int\limits_{\gamma_1 \cup \gamma_2 \cup \gamma_3}
\int\limits_{\Gamma_1 \cup \Gamma_2 \cup \Gamma_3}
\frac{G(x_0,0,w_0)h(x_0)w_0 X }{\left( 1- \frac w{w_0}\right)^2}\,
 x^{-n-1} w^{-k-1}\, dx\, dw\\
&\qquad= -\frac{G(x_0,0,w_0) h(x_0)}{2 \sqrt{\pi}} \,
g(x_0)^{k-1} x_0^{-n} k\, n^{-\frac 32}
\left( 1 + O\left( \frac 1k \right) \right),
\end{align*}
compare with the methods of Flajolet and Odlyzko \cite{FO}. We just
remark that $x^{-n}$ and $w^{-k}$ are replaced by
\[
x^{-n} = x_0^{-n}e^{-t + O(t^2/n)} \quad\mbox{and}\quad w^{-k} =
w_0^{-k}e^{-s + O(s^2/k)},
\]
so that one can use Hankel's representation of $1/\Gamma(s)$ to
evaluate the resulting integrals asymptotically.

Finally, the remainder term provides an error term of the form
\[
\int\limits_{\gamma_1 \cup \gamma_2 \cup \gamma_3}
\int\limits_{\Gamma_1 \cup \Gamma_2 \cup \Gamma_3}
O\left( \frac{|X|}{\left| 1- \frac w{w_0} \right| } \right)
 |x|^{-n-1} |w|^{-k-1}\, |dx|\, |dw|
= O\left( w_0^{-k} x_0^{-n}  n^{-\frac 32}\right),
\]
compare again with \cite{FO}.

If $x\in \gamma_1 \cup \gamma_2 \cup \gamma_3$ and $w\in \Gamma_4$
we can argue in a similar way. We use the expansions
\begin{align}
\frac 1{1 - y(x)w} &=
\frac 1{1- \frac w{w_0}} +
O\left( \frac{|X|}{\left| 1- \frac w{w_0}\right|^2} \right), \label{eqden1}\\
G(x,X,w) &= G_0(w) + O(|X|), \label{eqden2}
\end{align}
and the bounds $x^{-n} = O( x_0^{-n})$ and $w^{-k} = O(w_0^{-k}e^{-(\log k)^2})$
to observe that
\begin{align*}
&\frac{1}{(2\pi i)^2} \int\limits_{\gamma_1 \cup \gamma_2 \cup \gamma_3}
\int\limits_{\Gamma_4}
\left( \frac{G_0(w)}{1- \frac w{w_0}} +
O\left( \frac {|X|}{ \left| 1- \frac w{w_0} \right|} \right) \right)
x^{-n-1} w^{-k-1}\, dx\, dw \\
&\qquad = O\left( w_0^{-k} x_0^{-n} k e^{-(\log k)^2} n^{-\frac 32}\right).
\end{align*}
Note that the first term does not depend on $x$ and does not contribute
if $n\ge 1$.

Next suppose that $x\in \gamma_4$ and $w\in \Gamma_1 \cup \Gamma_2 \cup \Gamma_3$.
In this case we need not be that precise.
We can use the estimates $G(x,X,w) = O(1)$ and $|1-y(x)w| \ge c/k \ge c'/\log n$
(for certain positive constants $c,c'$). Furthermore we have
$x^{-n} = O( x_0^{-n} e^{-(\log n)^2})$ and $w^{-k} = O(w_0^{-k})$.
Hence the corresponding integral can be estimated by
\[
\frac{1}{(2\pi i)^2} \int\limits_{\gamma_4} \int\limits_{\Gamma_1
\cup \Gamma_2 \cup \Gamma_3} \frac{G(x,X,w)}{1-y(x)w} x^{-n-1}
w^{-k}\, dx\, dw = O\left( x_0^{-n} w_0^{-k} \log n\,  e^{-(\log
n)^2}) \right),
\]
which is negligible compared to the asymptotic leading term.

Finally, if $x\in \gamma_4$ and $w\in \Gamma_4$, the corresponding
integral just provides an error term of the form
\[
O\left( x_0^{-n} w_0^{-k} \log n\,  e^{-(\log n)^2-(\log k)^2)}
\right),
\]
which is again negligible. This proves the asymptotic expansion
(\ref{eqLe31-asmp-exp}).

In order to show the upper bound (\ref{eqLe32}) we use again
Cauchy's formula for $x \in \gamma_1 \cup \gamma_2 \cup \gamma_3 \cup \gamma_4$
(as above) and $|w| = w_0(1-\delta)$ for some $\delta> 0$.
If $x \in \gamma_1 \cup \gamma_2 \cup \gamma_3$ we use the
expansions (\ref{eqden1}) and (\ref{eqden2}) to obtain
\begin{align*}
&\frac{1}{(2\pi i)^2} \int\limits_{\gamma_1 \cup \gamma_2 \cup \gamma_3}
\int\limits_{|w| = w_0(1-\delta)}
\left( \frac{G_0(w)}{1- \frac w{w_0}} +
O\left( \frac {|X|}{ \left| 1- \frac w{w_0} \right|} \right) \right)
x^{-n-1} w^{-k-1}\, dx\, dw \\
&\qquad = O\left( w_0^{-k}(1-\delta)^{-k} x_0^{-n} n^{-\frac 32}\right).
\end{align*}
The integral over  $x\in \gamma_4$ and $|w| = w_0(1-\delta)$ can
be estimated directly (and similarly to the above), which leads
to an additional error term of the form
\[
O\left( x_0^{-n} w_0^{-k}(1-\delta)^{-k} e^{-(\log n)^2} \right).
\]
This completes the proof of (\ref{eqLe32}).

Now suppose that $g(x_0)< 1$. Then the generating function
of $f_n = \sum_k f_{n,k}$ is given by
\[
\sum_{n\ge 0} f_n x^n = f(x,1) = \frac{G(x,X,1)}{1-y(x)}.
\]
By a local expansion it follows that
\[
\frac{G(x,X,1)}{1-y(x)} = \frac{G(x_0,0,1)}{1-g(x_0)} -
\frac{h(x_0)G(x_0,0,1) - (1-g(x_0))G_v(x_0,0,1)} {(1-g(x_0))^2} X +
O(X^2),
\]
which induces an asymptotic expansion for $f_n$ of the form
\[
f_n \sim \frac{h(x_0)G(x_0,0,1) - (1-g(x_0))G_v(x_0,0,1)} {2
\sqrt\pi (1-g(x_0))^2} x_0^{-n} n^{-3/2}.
\]
Similarly we can derive an asymptotic expansion for
\[
\sum_k f_{n,k}w^k = [x^n] f(x,w),
\]
which shows that the limit
\[
\lim_{n\to\infty} \sum_{k} \frac{f_{n,k}}{f_n} w^k
\]
exists uniformly for $|w|\le 1$. This also shows the existence of the
limits $\overline d_k= \lim_{n\to\infty} f_{n,k}/f_n$.
\end{proof}

\begin{proof}[Proof of Lemma \ref{Le4}]
The proof of Lemma~\ref{Le4} is very close to that of
Lemma~\ref{Le3}. The essential observation is an asymptotic
representation of $f(x,w,t)$ of the form
\begin{align*}
f(x,w,t) &= \frac{G(x_0,0,w_0,w_0)}{X \left( 1 - \frac w{w_0}\right)^2
\left( 1 - \frac t{w_0}\right)^2 }
- \frac{2 h(x_0) G_0(w,t)}{1 - \frac w{w_0}} \\
&- \frac{2 h(x_0) G_0(w,t)}{1 - \frac t{w_0}} + O\left(
\frac{X}{|w-w_0|} + \frac{X}{|t-w_0|} \right),
\end{align*}
for $(x,w,t)$ close to the singularity $(x_0,w_0,w_0)$, and the fact
that the first term is the asymptotic leading one.
\end{proof}

\begin{proof}[Proof of Lemma~\ref{Le5}]
The main observation is that $f(x,w)$ can be approximated by
\begin{align*}
f(x,w) &= G(x_0,0,w_0)\, \exp\left( \frac{H(x_0,0,w)}{1 - \frac w{w_0}} \right) \\
&\quad \times \Biggl( 1  + \Biggl( \frac{G(x_0,0,w)}{G_v(x_0,0,w)}
+  \frac{H_v(x_0,0,w)}{1 - \frac w{w_0}}  - \frac{H(x_0,0,w)h(x_0)w_0 }
{\left(1 - \frac w{w_0} \right)^2}
\Biggr)X \\
& \qquad \qquad  + O\left( \frac{X^2}{|w-w_0|^3} \right) \Biggr)
\end{align*}
if $x$ and $w$ are close to their singularities $x_0$ and $w_0$. Hence, the term
\[
- G(x_0,0,w_0)\,
\exp\left( \frac{H(x_0,0,w)}{1 - \frac w{w_0}} \right)
\frac{h(x_0)H(x_0,0,w_0)w_0 X}{\left(1 - \frac w{w_0} \right)^2}
\]
leads to the asymptotic expansion as claimed. The main difference
with the proof of Lemma~\ref{Le3} is that the contour integration
with respect to $w$ is a circle $|w| = w_0(1- \eta)$, where $\eta
\sim c k^{-1/2}$ is chosen such that $w_0(1-\eta)$ becomes a saddle
point of the integrand; compare with Hayman's method \cite{Hayman}
and with \cite{BPS}. All error terms can be easily estimated.
\end{proof}

\begin{proof}[Proof of Lemma~\ref{Le6}]
The proof is (more or less) a direct combination of the methods used
in the proof of Lemma~\ref{Le3} and \ref{Le5}. The main observation
here is that the asymptotic leading term of $f(x,w,t)$ is of the
form
\[
\frac{G(x_0,0,w_0,w_0)}{X}\, \frac{\displaystyle\exp\left(
\frac{H(x_0,0,w)}{1- w/{w_0}} + \frac{H(x_0,0,t)}{1-  t/{w_0}}
\right)} {\left( 1- w/{w_0}\right)^2\left( 1- t/{w_0}\right)^2}.
\]
\end{proof}

\begin{proof}[Proof of Lemma~\ref{Le7}]
We can neglect the function $G(x,X,w)$ since it is analytic in $w$
around the critical point $(x_0,w_0)$.

The main observation now it that the remaining part can be
represented (locally) as
\begin{align*}
& H(x,X,w)  \left( 1- \frac w{w_0} + h(x_0)w_0 X + O(X^2) \right)^{3/2} \\
&= H(x,X,w) \left( 1- \frac w{w_0}\right)^{3/2}
\left( 1 + \frac{(3/2)h(x_0)w_0 X}{1 - \frac w{w_0}} + O\left( \frac{X^2}{|w-w_0|} \right)\right).
\end{align*}
Hence, the main contribution comes from the term
\[
\frac 32 h(x_0)w_0H(x_0,0,w_0) \left( 1- \frac w{w_0}\right)^{1/2}  \left( 1- \frac x{x_0}\right)^{1/2}.
\]
\end{proof}

\begin{proof}[Prof of Lemma \ref{Le8}]
The proof is an easy combination of the proofs of Lemma~\ref{Le7} and \ref{Le4}.
The asymptotic leading term of $f(x,w,t)$ is given by
\[
\frac{G_4(x_0,0,w_0,w_0)}{X} \left( 1 - \frac w{w_0} \right)^{1/2}
\left( 1 - \frac t{w_0} \right)^{1/2}.
\]
\end{proof}


\begin{thebibliography}{99}

\bibitem{BPS}
N. Bernasconi, K. Panagiotou, A. Steger, The degree sequence of
random graphs from subcritical classes,   Combin. Probab. Comput. 18
(2009),  647--681.


\bibitem{BGKN-series-parallel}
M. Bodirsky, O. Gim\'enez, M. Kang and M. Noy, Enumeration and limit
laws for series-parallel graphs. Europ. J. Combin. 8 (2007), 2091--2105.

\bibitem{DrEJC}
M. Drmota,
A bivariate asymptotic expansion of coefficients of powers of generating functions,
European J. Combin. 15 (1994), no. 2, 139--152

\bibitem{Drm-randomtrees}
M. Drmota,
Random Trees, Springer, Wien-New York, 2009.

\bibitem{DGN1}
M. Drmota, O. Gim\'enez, M. Noy, Vertices of given degree in
series-parallel graphs, Random Structures Algorithms 36 (2010),
273--314.

\bibitem{DGN2}
M. Drmota, O. Gim\'enez, M. Noy, Degree distribution in random
planar graphs, arXiv:0911.4331

\bibitem{DGN3}
M. Drmota, O. Gim\'enez, M. and Noy,
The maximum degree of planar graphs II,
manuscript in preparation.

\bibitem{FO}
P. Flajolet and A. Odlyzko, Singularity analysis of generating functions,
SIAM J. Discrete Math.
3 (1990), 216-- 240.

\bibitem{FS} P.~Flajolet, R.~Sedgewick,
\emph{Analytic Combinatorics}, Cambridge University Press, Cambridge (2009).

\bibitem{gao-wormald}
Z. Gao and   N. C. Wormald, The distribution of the maximum vertex
degree in random planar maps.  J. Combin. Theory Ser. A  89
(2000),  201--230.

\bibitem{Hayman}
W. K. Hayman, A generalisation of Stirling's formula, J. Reine Angew. Math. 196 (1956),
67--95.

\bibitem{MR} C. McDiarmid and B. Reed,
On the maximum degree of a random planar graph,
Combin. Probab. Comput. 17 (2008), 591--601.


\end{thebibliography}
\end{document}